\documentclass[12pt,reqno]{amsart}
\usepackage{enumerate}
\usepackage{amsrefs}
\usepackage{amsfonts, amsmath, amssymb, amscd, amsthm, bm, cancel}
\usepackage{url}
\usepackage{graphicx}
\usepackage[
linktocpage=true,colorlinks,citecolor=magenta,linkcolor=blue,urlcolor=magenta]{hyperref}
\usepackage{multicol}
\usepackage{comment}
\usepackage[margin=1in]{geometry}
\newtheorem{thm}{Theorem}[section]

 \newtheorem{cor}{Corollary}[section]

  \newtheorem{lem}{Lemma}[section]

 \newtheorem{prop}{Proposition}[section]
 \theoremstyle{definition}
 
  \newtheorem*{ack}{Acknowledgments}
 \theoremstyle{remark}
\newtheorem{rem}{Remark}[section]

 \numberwithin{equation}{section}

\allowdisplaybreaks

\newcommand{\f}{\left(}
\renewcommand{\r}{\right)}

\renewcommand{\a}{\alpha}
\renewcommand{\b}{\beta}

\renewcommand{\d}{\delta}

\renewcommand{\k}{\kappa}

\renewcommand{\o}{\omega}

\renewcommand{\t}{\theta}
\newcommand{\s}{\sigma}

\newcommand{\metric}[2]{\ensuremath{\langle #1, #2\rangle}}  

\begin{document}
\title{A Flow approach to the prescribed Gaussian curvature problem in $\mathbb{H}^{n+1}$}

\author{Haizhong Li}
\address{Department of Mathematical Sciences, Tsinghua University, Beijing 100084, P.R. China}
\email{\href{mailto:lihz@tsinghua.edu.cn}{lihz@tsinghua.edu.cn}}

\author{RUIJIA ZHANG}
\address{Department of Mathematical Sciences, Tsinghua University, Beijing 100084, P.R. China}
\email{\href{mailto:zhangrj17@mails.tsinghua.edu.cn}{zhangrj17@mails.tsinghua.edu.cn}}

\keywords{Curvature flow, Monotonicity, Asymptotic behaviour, Hyperbolic space}
\subjclass[2010]{35K55, 53E10}


\begin{abstract}
In this paper, we study the following prescribed Gaussian curvature problem
$$K=\frac{\tilde{f}(\theta)}{\phi(\rho)^{\alpha-2}\sqrt{\phi(\rho)^2+|\overline{\nabla}\rho|^2}},$$
a generalization of the Alexandrov problem ($\alpha=n+1$) in hyperbolic space, where $\tilde{f}$ is a smooth positive function on $\mathbb{S}^{n}$, $\rho$ is the radial function of the hypersurface, $\phi(\rho)=\sinh\rho$ and $K$ is the Gauss curvature. By a flow approach, we obtain the existence and uniqueness of solutions to the above equations when $\alpha\geq n+1$. Our argument provides a parabolic proof in smooth category for the Alexandrov problem in $\mathbb{H}^{n+1}$. We also consider the cases $2<\alpha\leq n+1$ under the evenness assumption of $\tilde{f}$ and prove the existence of solutions to the above equations. 
\end{abstract}

\maketitle

\section{introduction}\label{sec:1}
The Alexandrov problem proposed by A. D. Alexandrov \cite{A43} is of great significance to the study of convex bodies. It quests for the existence of closed convex hypersurfaces with prescribed volume element of the Gaussian image in the Euclidean space. Let $\hat{M}^n$ be the boundary of some convex domains containing a neighborhood of the origin in $\mathbb{R}^{n+1}$, which can be written as a radial graph over $\mathbb{S}^n$, i.e., $\hat{M}^n=\lbrace R(\t)=r(\t)\t|\t\in \mathbb{S}^n\rbrace$ with induced metric $\hat{g}_{ij}$. Denote by $\hat{\nu}(Y):\hat{M}^n\rightarrow\mathbb{S}^n$ the generalized Gauss map. For smooth $\hat{M}^n$, $\hat{\nu}(Y)$ is the unit outward normal vector at $Y\in \hat{M}^n$. The Alexandrov problem is to reconstruct $\hat{M}^n$ by given integral Gaussian curvature
\begin{align*}
    \mu(\o)=|\nu(R(\o))|
\end{align*}
 for a nonnegative completely additive function $\mu$ on the set of Borel subsets $\o$ of $\mathbb{S}^n$. Furthermore, if $\hat{M}^n$ is at least $C^2$, then
\begin{align*}
   |\nu(R(\o))|=\int_{R(\o)} \hat{K}\mathrm{d}v_{\hat{M}^n}=\int_{\o}\hat{K}\sqrt{\det(\hat{g}_{ij})}\mathrm{d}\t_{\mathbb{S}^n}
\end{align*}
 where $\hat{K}$ is the Gauss curvature of $\hat{M}^n$ and $\mathrm{d}\t_{\mathbb{S}^n}$ is the standard measure on $\mathbb{S}^{n}$. If $\mu$ is given by integrating a function, we write it as 
 \begin{align*}
     \mu(\o)=\int_{\o}\tilde{f}\mathrm{d}\t_{\mathbb{S}^n}.
 \end{align*}
 Then the Alexandrov problem can be reduced to the following fully nonlinear partial differential equation
 \begin{align*}
     \hat{K}=\frac{\tilde{f}}{r^{n-1}\sqrt{r^2+|\bar{\nabla}r|^2}}.
 \end{align*}
The existence of regular solutions to this equation was solved by Pogorelev \cite{P73} for surfaces and Oliker \cite{O83} for higher dimensional cases.
For more related interesting studies of the Alexandrov problem in $\mathbb{R}^{n+1}$, one can refer to \cites{GL97, T90}.

Naturally, similar prescribed Gaussian curvature problems of hypersurface $M^n$ in $\mathbb{H}^{n+1}$ were studied in \cites{O83,O89}, where the given function defined on $M^n$ need constrained conditions to ensure $C^0$ estimate. Recently, Yang \cite{Y20} studied the Alexandrov problem in the hyperbolic space. He considered the hypersurface $M^n$ in $\mathbb{H}^{n+1}$ whose Gauss curvature measures were prescribed via a radial map.
Let $M^n$ be the boundary of some convex body in $\mathbb{H}^{n+1}$ enclosing the origin. we can parametrize it as a graph of the radial function $\rho(\theta)$, such that $M^n = \lbrace R(\t)=(\rho(\theta), \theta):\rho:\mathbb{S}^n\rightarrow\mathbb{R}^+,\theta \in \mathbb{S}^n \rbrace$. If $M^n$ is at least $C^2$, similar to the Euclidean space, then we can define the prescribed Gaussian curvature measure problem by
\begin{align*}
      \int_{R(\o)} K\mathrm{d}v_{M^n}=\int_{\o}\tilde{f} \mathrm{d}\t_{\mathbb{S}^n}
\end{align*}
where $\tilde{f}$ is a given positive function on $\mathbb{S}^n$. By the coordinate transformation, we write the both sides of the above integrals on any Borel set $\o$ of $\mathbb{S}^n$ as 
\begin{align*}
    \int_{\o}K\sqrt{\det(g_{ij})}\mathrm{d}\t_{\mathbb{S}^n} =\int_{\o}\tilde{f} \mathrm{d}\t_{\mathbb{S}^n}.
\end{align*}
Then this curvature measure problem is reduced to the following fully nonlinear PDE
\begin{align}\label{alex n+1}
K=\frac{\tilde{f}(\theta)}{\phi(\rho)^{n-1}\sqrt{\phi(\rho)^2+|\overline{\nabla}\rho|^2}}\quad {\rm on} \ \mathbb{S}^n.
\end{align}
Yang proved the existence and uniqueness of \eqref{alex n+1} with the condition $\inf \tilde{f}>1$. In his study \cite{Y20}, he mentioned that whether the solution to \eqref{alex n+1} exists with $\tilde{f}$ endowed with other geometric conditions is still an open question. 
Note that the condition $\inf\tilde{f}>1$ in  \cite{Y20} is only used in the $C^0$ estimate. In the studies of prescribed curvature measure problem, $C^0$ estimates are difficult in most cases, but are also more geometric. In Theorem \ref{thm2} and Corollary \ref{cor 2}, we weaken the condition of $\tilde{f}$ and prove the existence of solutions to \eqref{alex n+1} under the evenness assumption by deriving a delicate $C^0$ estimate. Besides, in this paper we consider the following more general prescribed Gaussian curvature problems which correspond to the following fully nonlinear PDE
\begin{align}\label{alex2}
K=\frac{\tilde{f}(\theta)}{\phi(\rho)^{\a-2}\sqrt{\phi(\rho)^2+|\overline{\nabla}\rho|^2}}\quad {\rm on} \ \mathbb{S}^n.
\end{align}

Motivated by the flow studied by Li-Sheng-Wang \cite{LSW20a} in the Euclidean space, we provide a curvature flow approach to \eqref{alex2} in hyperbolic space. Write $f(\theta)=\tilde{f}(\theta)^{-1}$.
Let $M_0$ be a smooth closed uniformly convex hypersurface in $\mathbb{H}^{n+1}$ enclosing the origin. In this paper, we study the following sort of flow
\begin{equation}\label{flow-s}
\left\{\begin{aligned}
\frac{\partial}{\partial t}X(x,t)=& -\phi(\rho)^{\a}f(\theta)K(x,t) \nu(x,t)+V(x,t), \\
X(\cdot,0)=& X_0(\cdot),
\end{aligned}\right.
\end{equation}
where $\a\geq n+1$ is a constant. Here we regard $\mathbb{H}^{n+1}$ as a warped product space. Any point $X\in\mathbb{H}^{n+1}$ can be parametrized by $X=(\rho,\t)\in \mathbb{R}^+\times\mathbb{S}^n$. Then $f$ is a smooth positive function defined on $\mathbb{S}^{n}$, $\phi(\rho)=\sinh\rho$, $K$ is the Gauss curvature of the flow hypersurface $M_t$, $\nu$ is the unit outward normal at $X(x,t)$ and $V=\sinh \rho \ \partial_{\rho}$ is a conformal Killing vector field on $\mathbb{H}^{n+1}$. Equivalently, up to a tangential diffeomorphism the flow \eqref{flow-s} can be written as follows:
\begin{align}
\partial_t X=\f-\phi(\rho)^{\a}f(\theta)K(x,t) +u(x,t)\r\nu(x,t),
\end{align}
where $u=\metric{V}{\nu}$ is the support function of $M_t$.
Note that the following elliptic equation
\begin{align}\label{alex}
\phi(\rho)^{\a}K=\tilde{f}(\theta)u\quad {\rm on} \  \mathbb{S}^n
\end{align}
(that is exactly \eqref{alex2}) remains invariant under  \eqref{flow-s}. 

In this paper, we prove the following results.
\begin{thm}\label{thm1}
Let $M_0$ be a smooth, closed, uniformly convex hypersurface in hyperbolic space $\mathbb{H}^{n+1}$ enclosing the origin. Suppose that $f$ is a smooth positive function on $\mathbb{S}^{n}$. If
\begin{itemize}
\item[(\romannumeral1)]$\a>n+1$
or
\item[(\romannumeral2)]$\a=n+1$ and $f<1$,
\end{itemize}
then the flow \eqref{flow-s} has a unique smooth uniformly convex solution $M_t$ for all time $t>0$. When $t\rightarrow\infty$, $M_t$ converges smoothly to the unique smooth solution of \eqref{alex2}.
\end{thm}

\begin{cor}\label{cor 1}
Suppose that $\tilde{f}$ is a smooth positive function on $\mathbb{S}^{n}$. Then there is a unique smooth uniformly convex hypersurface $M^n$ in $\mathbb{H}^{n+1}$, such that it satisfies \eqref{alex2} under one of the following two assumptions,
\begin{itemize}
\item[(\romannumeral1)]$\a>n+1$ or,
\item[(\romannumeral2)]$\a=n+1$ and $\inf \tilde{f}>1$.
\end{itemize}
\end{cor}
\begin{rem}
Case $(\romannumeral2)$ in Corollary \ref{cor 1} was proved by Fengrui Yang in \cite[Theorem 6]{Y20}. 
\end{rem}

\begin{thm}\label{thm2}
   Let $M_0$ be a smooth, closed, uniformly convex and origin-symmetric hypersurface in hyperbolic space $\mathbb{H}^{n+1}$. Assume $\a= n+1$, and $f$ is a smooth positive even function on $\mathbb{S}^{n}$ satisfying $\int_{\mathbb{S}^n} f(\theta)^{-1}\mathrm{d}\theta_{\mathbb{S}^n}>|S^n|$. Then the flow \eqref{flow-s} has a smooth uniformly convex solution for all time $t>0$, and converges smoothly to the unique smooth even solution of \eqref{alex n+1}.
\end{thm}
\begin{cor}\label{cor 2}
Suppose $\tilde{f}$ is a smooth positive even function on $\mathbb{S}^{n}$ satisfying $\int_{\mathbb{S}^n} \tilde{f}(\theta) \mathrm{d}\theta_{\mathbb{S}^n}>|S^n|$. Then there is a unique smooth uniformly convex and origin-symmetric hypersurface $M^n$ in $\mathbb{H}^{n+1}$, such that it satisfies \eqref{alex n+1}. 
\end{cor}
\begin{rem}
In Corollary \ref{cor 2}, we weaken the condition of $\tilde{f}$ in Case $(\romannumeral2)$ of Theorem \ref{thm2}, from $\inf \tilde{f}>1$ to $\int \tilde{f}(\theta) \mathrm{d}\theta_{\mathbb{S}^n}>|S^n|$ and obtain the existence of the even Alexandrov problem in hyperbolic space.
\end{rem}

When $2<\a\leq n+1$, we consider the following flow 
\begin{equation}\label{flow-n}
\left\{\begin{aligned}
\frac{\partial}{\partial t}X(x,t)=& -\phi(\rho)^{\a}f(\theta)K(x,t) \nu(x,t)+\eta(t)V(x,t), \\
X(\cdot,0)=& X_0(\cdot),
\end{aligned}\right.
\end{equation}
where $\eta(t)=\frac{\displaystyle\int_{\mathbb{S}^n}\frac{K}{u}\phi^{n+1}\mathrm{d}\theta_{\mathbb{S}^n}}{\displaystyle\int_{\mathbb{S}^n}\phi^{n+1-\a}f^{-1}\mathrm{d}\theta_{\mathbb{S}^n}}$ and obtain the existence of the following equation 
\begin{align}\label{alex'}
\phi(\rho)^{\a}K=c \tilde{f}(\theta)u\quad {\rm on} \ \mathbb{S}^n.
\end{align}
\begin{thm}\label{thm3}
   Let $M_0$ be as in Theorem \ref{thm2}. Assume $2<\a\leq n+1$, and $f$ is a smooth positive even function on $\mathbb{S}^{n}$. Then the flow \eqref{flow-n} has a smooth uniformly convex solution for all time $t>0$. When $t\rightarrow\infty$, $M_t$ converges smoothly to the smooth solution of \eqref{alex'} for some positive constant $c$ in a subsequence.
\end{thm}
\begin{cor}\label{cor 3}
Suppose that $\tilde{f}$ is a smooth positive even function on $\mathbb{S}^{n}$. If $2<\a < n+1$, then there is a smooth uniformly convex and origin-symmetric hypersurface $M^n$ in $\mathbb{H}^{n+1}$, such that it satisfies \eqref{alex'} for some positive constant $c$. 
\end{cor}

Curvature flows in hyperbolic space have been studied extensively in recent years. In these studies, constrained flows were introduced to prove geometric inequalities, see, e.g., \cites{ACW18,Hu-Li-Wei2020,LWX14,SX19,WX14}. Convergence results for inverse curvature flows were obtained in \cites{G11,LZ19,S15,S'15,WWZ20}. Furthermore, volume preserving curvature flows in hyperbolic space have been studied, see for example \cites{AW18,M12}. All these flows in hyperbolic space have the same limiting shape in common, i.e., they all become round. When $f\equiv 1$, case $(\romannumeral1)$ in Theorem \ref{thm1} was proved by Fang Hong in \cite{F21}, in which he proved that the flow \eqref{flow-s} converges smoothly to a geodesic sphere. For the first time, we introduce the function $f$ in flows \eqref{flow-s}, \eqref{flow-n} and derive the convergence results of solutions to \eqref{alex2}, \eqref{alex'}, which build a bridge between curvature flows in hyperbolic space and solutions to the elliptic equation. By using the Klein model (see also in \cites{AW18,CH21,W19}), we project the hyperbolic flow \eqref{flow-s} to the Euclidean space and obtain the projection flow \eqref{u eu}. We discover the monotone function \eqref{Q} along \eqref{u eu} and derive the asymptotic convergence result of \eqref{flow-s}. For $\a\leq n+1$, we design the flow \eqref{flow-n} and deduce the convergence result by deriving a delicate $C^0$ estimate.

This paper is organized as follows. In Section \ref{sec:2}, we collect some properties of star-shaped hypersurfaces in hyperbolic space, derive some evolution equations of various geometric quantities along \eqref{flow-s}, \eqref{flow-n} and show that the flows can be reduced to a scalar parabolic PDE for the radial function. In Section \ref{sec:3}, we prove $C^0$, $C^1$ estimates when $\a\geq n+1$ and show that the hypersurface preserves star-shaped along the flow \eqref{flow-s}. In Section \ref{sec:4}, we obtain the uniform bound of the Gauss curvature $K$ which implies the short time existence of the flow \eqref{flow-s}. By using a new auxiliary function, we obtain the uniform bound of the principal curvatures of $M_t$ and establish the a priori estimates for the long time existence of \eqref{flow-s}. In section \ref{sec:5}, we study the asymptotic behaviour of the flow \eqref{flow-s} by projecting $M_t$ to the Euclidean space, prove the uniqueness of the solution to \eqref{alex2} and complete the proof of Theorem \ref{thm1}. In section \ref{sec:6}, we complete the proof of Theorem \ref{thm2} by deriving a delicate $C^0$ estimate. In Section \ref{sec:7}, we study the normalized flow \eqref{flow-n} under the evenness assumption when $2<\a\leq n+1$ and complete the proof of Theorem \ref{thm3}. 
%

\begin{ack}
	The authors were partially supported by NSFC grant No.11831005 and NSFC grant No.12126405. The authors would like to thank Professor Xianfeng Wang and Professor Yong Wei for helpful discussions.
\end{ack}

\section{Preliminaries}\label{sec:2}
In this paper, we fix a point $o\in\mathbb{H}^{n+1}$ and consider the polar geodesic coordinates centered at $o$ and regard $\mathbb{H}^{n+1}$ as a warped product space $[0,+\infty)\times \mathbb{S}^n$ equipped with Riemannian metric 
\begin{align*}
g_{\mathbb{H}^{n+1}}=d\rho^2+\phi(\rho)^2g_{\mathbb{S}^n}
\end{align*}
where $\phi(\rho)=\sinh\rho$ and $g_{\mathbb{S}^n}$ is the standard metric on the unit sphere $\mathbb{S}^n$. Denote 
\begin{align}
\Phi(\rho)=\int_0^{\rho}\sinh s \ \mathrm{d}s=\phi'(\rho)-1.
\end{align}
The conformal Killing vector field can be written as $V=D\Phi=\phi\partial\rho$ and $D^2\Phi=\phi'g_{\mathbb{S}^n}$. In particular, $D V=\phi'g_{\mathbb{H}^{n+1}}$.

\subsection{Hypersurfaces in hyperbolic space}
Let $M^n$ be a closed hypersurface in $\mathbb{H}^{n+1}$ and $\lbrace x^1,\cdots,x^n\rbrace$ be a local coordinate system of $M^n$. We regard $\nu$ as the unit outward normal vector field of $M^n$. We denote by the induced metric $g_{ij}=g(X_i,X_j)$ and the second fundamental form $h_{ij}=h(X_i,X_j)$ of $M^n$, where the second fundamental form is defined by $h(X,Y)=\metric{\nabla_X \nu}{Y}$ with any two tangent vector fields $X,Y\in TM^n$. The Weingarten matrix is regarded as $\mathcal{W}=\lbrace h_i{}^j\rbrace=\lbrace h_{ik}g^{kj}\rbrace$, where $\lbrace g^{ij}\rbrace$ is the inverse matrix of $\lbrace g_{ij}\rbrace$. The principal curvature $\k=(\k_1,\cdots,\k_n)$ of $M^n$ are eigenvalues of $\mathcal{W}$. Let $f(\k)$ be a symmetric function of the principal curvatures $\k = (\k_1, \k_2, \cdots, \k_n)$. There exists a function $\mathcal{F}(\mathcal{W})$ defined on the Weingarten matrix, such that $F (\mathcal{W})=f(\k)$.
Since $h_i{}^j = \sum_k h_{ik} g^{kj}$, $\mathcal{F}$ can be viewed as a function $\hat{\mathcal{F}}(h_{ij}, g_{ij})$ defined on the second fundamental form $\lbrace h_{ij}\rbrace$ and the metric $\lbrace g_{ij}\rbrace$. In the subsequent article, we denote
\begin{align*}
\dot{ \mathcal{F} }^{pq}(\mathcal{W}): =&\frac{\partial \hat{\mathcal{F}}}{\partial h_{pq}}(h_{ij}, g_{ij}),\quad
\ddot{\mathcal{F}}^{pq,rs}(\mathcal{W}): =\frac{\partial ^2 \hat{\mathcal{F}} }{\partial h_{pq}\partial h_{rs}}(h_{ij}, g_{ij}).
\end{align*}

Here we collect some formulas of hypersurface in hyperbolic space (see \cites{GL15, HL21}). 
\begin{lem} 
Let $(M^n,g)$ be a smooth hypersurface in $\mathbb{H}^{n+1}$. Then we have
\begin{align}\label{Phiij}
\nabla_i \Phi=\metric{V}{X_i},\quad \nabla_j\nabla_i\Phi=\phi'g_{ij}-uh_{ij}.
\end{align}
The support function $u=\metric{V}{\nu}$ satisfies
\begin{align}\label{uij}
\nabla_i u=\metric{V}{X_k}h_i{}^k,\quad \nabla_j\nabla_i u=\metric{V}{\nabla h_{ij}}+\phi'h_{ij}-uh_i{}^kh_{kj}
\end{align}
where $\nabla$ is the levi-Civita connection on $M^n$ with respect to the induced metric and $\lbrace X_1,\cdots,X_n\rbrace$ is a basis of the tangent space of $M^n$.
\end{lem}
Then we have the first and second derivatives of the distance function $\rho$.
\begin{cor}
\begin{align}\label{rhoij}
\nabla_i\rho=\frac{\metric{V}{X_i}}{\phi},\quad \nabla_j\nabla_i\rho=\frac{\phi'}{\phi}(g_{ij}-\nabla_j\rho\nabla_i\rho)-\frac{u h_{ij}}{\phi}.
\end{align}
\end{cor}
\begin{proof}
Observe that
\begin{align*}
\nabla_i \Phi=\phi\nabla_i\rho,\quad \nabla_j\nabla_i\Phi=\phi\nabla_j\nabla_i\rho+\phi'\nabla_j\rho\nabla_i\rho.
\end{align*}
Combing with \eqref{Phiij} we get \eqref{rhoij} by a direct calculation.
\end{proof}

%
%

\subsection{Evolution equations}
 For convenience, we consider the following flow
\begin{equation}\label{s2:flow-entire}
\frac{\partial}{\partial t}X(x,t)= -\Theta \nu(x,t)+ \tilde{\eta}(t)V
\end{equation}
where $\Theta=\phi(\rho)^{\a}f(\theta)K$ and the global term $\tilde{\eta}(t)$ is a function of time $t$. For $\tilde{\eta}(t)\equiv 1$, \eqref{s2:flow-entire} is the flow \eqref{flow-s}; for $\tilde{\eta}(t)=\eta(t)$, \eqref{s2:flow-entire} is the flow \eqref{flow-n}.

	
%

\begin{lem}
	Along the flow \eqref{s2:flow-entire}, we have the following evolution equations (also see \cites{WWZ20, F21}). The induced metric evolves by 
\begin{align}\label{g ev}
\frac{\partial}{\partial t}g_{ij}=&-2\Theta h_{ij}+2 \phi' \tilde{\eta}(t) g_{ij}.
\end{align}
The support function evolves by
\begin{align}\label{u ev}
\frac{\partial}{\partial  t}u=&-\phi'\Theta+\phi'\tilde{\eta}(t) u+\langle V,\nabla \Theta \rangle.
\end{align}
The second fundamental form evolves by
\begin{align}\label{evolution h}
\frac{\partial}{\partial t} h_i{}^j=\nabla_i\nabla^j \Theta+\Theta h_i{}^kh_k{}^j-\tilde{\eta}(t)\phi'h_i{}^j+(\tilde{\eta}(t)u-\Theta)\delta_i{}^j
\end{align}
where $\nabla$ is the Levi-Civita connection of the induced metric on $M_t$.	
\end{lem}
\begin{proof}
By a direct calculation, we have
\begin{align*}
\frac{\partial}{\partial t}g_{ij}=&{\partial_t} \metric{\partial_i X}{\partial_j X}\\
=&\metric{D_i\f-\Theta\nu+ \tilde{\eta}(t) V\r}{\partial_jX} +\metric{ \partial_iX}{D_j\f-\Theta\nu+ \tilde{\eta}(t) V\r} \\
=&-\Theta\f \langle D_i\nu,\partial_jX \rangle +\langle \partial_i X,D_j\nu \rangle \r+2 \tilde{\eta}(t) \phi'g_{ij}\\
=&-2\Theta h_{ij}+2 \tilde{\eta}(t)\phi' g_{ij}.
\end{align*}
Since $\partial_t \nu$ is tangential,
\begin{equation} \label{ev-nu}
\begin{aligned}
\frac{\partial}{\partial t} \nu=& \langle \partial_t\nu,\partial_jX \rangle g^{il}\partial_l X\\
=&-\langle \nu,\partial_j\f-\Theta\nu+ \tilde{\eta}(t) V\r \rangle g^{jl}\partial_l X\\
=& \partial_j \Theta g^{jl}\partial_l X
=\nabla\Theta.
\end{aligned}
\end{equation}
Using \eqref{ev-nu}, we obtain the evolution of the support function $u$ as follows:
\begin{align*}
\frac{\partial}{\partial t}u=&\partial_t\metric{ V}{\nu}
=\metric{ -\phi'\Theta\nu+  \phi'\tilde{\eta}(t)V}{\nu }+\metric{V}{\nabla\Theta}\\
=&-\phi'\Theta+  \phi'\tilde{\eta}(t)u+\metric{V}{\nabla\Theta}.
\end{align*}
Now we calculate the evolution of $h_{ij}$
\begin{align*}
\frac{\partial}{\partial t} h_{ij}=&-\partial_t \metric{ D_{\partial_i X} \partial_j X}{\nu} \\
=&- \metric{D_{\partial_i X}D_{\partial_j X}(-\Theta\nu+ \tilde{\eta}(t) V)}{\nu}-R^{\mathbb{H}^{n+1}}(\partial_i X, \partial_t X,\partial_j X,\nu)- \metric{D_{\partial_i X} \partial_j X}{\nabla\Theta}\\
=&\partial_i\partial_j\Theta-\Theta (h^2)_{ij}+\tilde{\eta}(t)\phi'h_{ij}+(\tilde{\eta}(t)u-\Theta) g_{ij}-\metric{\nabla_{\partial_i X} \partial_j X}{\nabla\Theta}\\
=&\nabla_i\nabla_j\Theta-\Theta  h_i{}^kh_{kj}+ \tilde{\eta}(t)\phi'h_{ij}+(\tilde{\eta}(t)u-\Theta) g_{ij}.
\end{align*}
From \eqref{g ev}, we have
\begin{align*}
\frac{\partial}{\partial t} g^{ij}=-g^{il}\f \partial_t g_{lm} \r g^{mj}=2\Theta h^{ij}-2 \tilde{\eta}(t)\phi'g^{ij}.
\end{align*}
Thus
\begin{equation}
\begin{aligned}\label{h_i^j}
\frac{\partial}{\partial t} h_i{}^j=&\partial_t h_{il}g^{lj}+h_{il}\partial_t g^{lj}\\
=&\nabla_i\nabla^j\Theta+\Theta h_i{}^kh_k{}^j- \tilde{\eta}(t)\phi' h_i{}^j+(\tilde{\eta}(t)u-\Theta)\delta_i{}^j.
\end{aligned}
\end{equation}
\end{proof}
\subsection{Parametrization by radial graph}\label{subsec 2}

For a closed star-shaped hypersurface $M^n \subset \mathbb{H}^{n+1}$, we can parametrize it as a graph of the radial function $\rho(\theta): \mathbb{S}^n\to \mathbb{R}$, i.e.,
	\begin{align*}
	M^n = \lbrace (\rho(\theta), \theta):\rho:\mathbb{S}^n\rightarrow\mathbb{R}^+,\quad \theta \in \mathbb{S}^n \rbrace
	\end{align*}
	where $\theta=(\theta^1,\cdots,\theta^n)$ is a local normal coordinate system of $\mathbb{S}^n$ and $\rho$ is a smooth function on $\mathbb{S}^n$. 
	Let $f_i = \overline{\nabla}_i f$, $f_{ij} = \overline{\nabla}^2_{ij} f$, where $\overline{\nabla}$ is the Levi-Civita connection on $\mathbb{S}^n$ with respect to the standard metric $g_{\mathbb{S}^n}$.\\
The tangent space of $M^n$ is spanned by (also see \cite{CLW18})
\begin{align}
X_i=\rho_i\partial_{\rho}+\partial_{\theta_i}
\end{align}
and the unit outward normal vector is 
\begin{align*}
\nu=\frac{\partial_{\rho}-\frac{\rho^i\partial_{\theta_i}}{\phi^2}}{w},
\end{align*}
where we set 
\begin{align}\label{w}
w=\sqrt{1+\frac{|\overline{\nabla} \rho|^2}{\phi^2}}.
\end{align}
Then the support function and the induced metric can be expressed as 
\begin{align}\label{u}
u=\frac{\phi^2}{\sqrt{\phi^2+|\overline{\nabla} \rho|^2}},
\end{align}
\begin{align}\label{gij}
g_{ij}=\phi^2\delta_{ij}+\rho_i\rho_j,\quad g^{ij}=\frac{1}{\phi^2}(\delta^{ij}-\frac{\rho_i\rho_j}{\phi^2+|\overline{\nabla} \rho|^2}).
\end{align}
The second fundamental form is given by
\begin{align}\label{hi_j}
h_{ij}=\frac{-\phi \rho_{ij}+2\phi'\rho_i\rho_j+\phi^2\phi'\delta_{ij}}{\sqrt{\phi^2+|\overline{\nabla} \rho|^2}}
\end{align}
and we have the Weingarten matrix
\begin{align}\label{hij}
h_i{}^j=\frac{1}{\phi^2\sqrt{\phi^2+|\overline{\nabla} \rho|^2}}\f\delta^{jk}-\frac{\rho_j\rho_k}{\phi^2+|\overline{\nabla} \rho|^2}\r\f-\phi \rho_{ki}+2\phi'\rho_k\rho_i+\phi^2\phi'\delta_{ki}\r.
\end{align}
Similar to \cite[p. 901]{LSW20a}, the flow \eqref{s2:flow-entire} can be written as a scalar parabolic PDE for the radial function 
	\begin{equation}\label{s1:flow-rn}
	\left\{\begin{aligned}
	\frac{\partial}{\partial t}\rho(\theta,t)=&-\phi(\rho)^{\a}f(\theta)wK +\tilde{\eta}(t)\phi(\theta,t),\quad \text{for }(\theta,t)\in\mathbb{S}^n\times[0,+\infty),\\
	\rho(\cdot,0)=& \rho_0(\cdot),
	\end{aligned}\right.
	\end{equation}
	where $w$ is the function defined in \eqref{w}.

\section{ $C^0$ and $C^1$ estimates}\label{sec:3}
In this section, we establish the $C^0$ and $C^1$ estimates of the flow \eqref{flow-s} for the proof of Theorem \ref{thm1}. Especially, we show that the flow hypersurface $M_t$ preserves star-shaped along \eqref{flow-s}.
\subsection{ $C^0$ estimate}\label{sec:3.1}
In this subsection, we will show that the radial function $\rho$ of \eqref{flow-s} has uniform bound under the assumption of Theorem \ref{thm1}.
\begin{lem}\label{C0-est}
	Let $\rho(\cdot, t)$ be a smooth, positive, uniformly convex solution to \eqref{s1:flow-rn} on $\mathbb{S}^n \times [0,T)$ provided $\tilde{\eta}(t)\equiv 1$. If $\a > n+1$, or $\a = n+1$ with $f<1$, then there is a positive constant $C$ depending only on $n$, $\max f$, $\inf f$ and the initial hypersurface $M_0$, such that
	\begin{align*}
	\frac{1}{C} \leq \rho(\cdot, t) \leq C, \quad \forall t \in [0,T).
	\end{align*}
\end{lem}

\begin{proof} 
Fix time $t$ and suppose that $\rho$ attains its maximum at point $(p_0,t)$. At $(p_0,t)$, we have $|\overline{\nabla} \rho|=0$ and $\rho_{ij}\leq 0$. From \eqref{hij},
\begin{align}\label{c01}
h_i{}^j\geq\frac{\phi'}{\phi}\delta_i{}^j.
\end{align}
Inserting \eqref{c01} into \eqref{s1:flow-rn}, we obtain 
\begin{align}\label{c0 upp}
\partial_t\rho\leq -\phi^{\a-n}\phi'^nf+\phi=\phi(-\phi^{\a-n-1}\phi'^nf+1)\leq \phi(-\phi^{\a-1}f+1),
\end{align}
where we use the fact $\frac{\phi'}{\phi}\geq 1$.
Since $\a\geq n+1$, if $\phi\leq (\min f)^{-\frac{1}{\a-1}}$ for all $t\geq 0$, then we obtain the uniform upper bound of $\rho$. If there is some $t_0$, such that $\phi>(\min f)^{-\frac{1}{\a-1}}$,  then we obtain $\rho\leq \max_{\mathbb{S}^n}\rho(\cdot,0)$ by \eqref{c0 upp}. Thus, $\rho$ obtains the uniform upper bound with respect to the positive lower bound of $f$ on $\mathbb{S}^n$ and the initial hypersurface.

Suppose $\rho$ attains its spatial minimum at point $(q_0,t)$. Similarly, at $(q_0,t)$, we have
\begin{align}
h_i{}^j\leq\frac{\phi'}{\phi}\delta_i{}^j
\end{align} 
and
\begin{align}\label{c0 low}
\partial_t\rho\geq -\phi^{\a-n}\phi'^nf+\phi=\phi(-\phi^{\a-n-1}\phi'^nf+1).
\end{align}
When $\a>n+1$, $\phi^{\a-n-1}\phi'^n=(\sinh\rho)^{\a-n-1}(\cosh\rho)^n\rightarrow 0$ as $\rho\rightarrow 0$. When $\a=n+1$, $\phi^{\a-n-1}\phi'^n=(\cosh\rho)^n\rightarrow 1$ as $\rho\rightarrow 0$. Hence if $\phi^{\a-n-1}\phi'^n\geq \frac{1}{\max f}$ for all $t\geq 0$, then we obtain the uniform lower bound of $\rho$ provided $\a>n+1$ or $\a=n+1$ with the assumption $f<1$. If there is some $t_0$, such that $\phi^{\a-n-1}\phi'^n<\frac{1}{\max f}$, we obtain $\rho\geq \min_{\mathbb{S}^n}\rho(\cdot,0)$ by \eqref{c0 low}. 
Thus $\rho$ has the uniform lower bound depending on $\min f$ and the initial hypersurface.
	\end{proof}
	
\subsection{ $C^1$ estimate}
In this subsection, we derive a uniform upper bound of the gradient of $\rho$ by using the approach in \cites{G06,WWZ20}.
\begin{lem}\label{c1}
Let $\rho(\cdot, t)$ be a smooth, positive, uniformly convex solution to \eqref{s1:flow-rn} on $\mathbb{S}^n \times [0,T)$. Based on the results of Lemma \ref{C0-est}, we have
\begin{equation*}
|\overline{\nabla}\rho|\leq C,
\end{equation*}
where $C$ only depends on the uniform upper and lower bounds of $\rho$.
\end{lem}
\begin{proof}
Fix some time $t$ and consider the auxiliary function $Q=\log w+\b\rho$, where we assume $\b=-2\tanh(\rho_{\min})$ from Lemma \ref{C0-est}. At the maximal point of $Q$, we have $\overline{\nabla} Q=0$, then 
\begin{align}\label{c1 cri}
\frac{\gamma_{li}\gamma_l}{w^2}+\b\rho_i=0,
\end{align}
where we regard $\gamma(\rho)=\log(1-\frac{2}{e^{\rho}+1})$ and $\frac{d\gamma}{d\rho}=\frac{1}{\phi}$.
Then \eqref{w}, \eqref{gij}, \eqref{hi_j}, \eqref{hij} becomes 
\begin{align}\label{g2}
g_{ij}=\phi^2(\delta_{ij}+\gamma_i \gamma_j),\quad g^{ij}=\frac{1}{\phi^2}(\delta^{ij}-\frac{\gamma_i \gamma_j}{w^2}),
\end{align}
\begin{align}\label{h2}
w=\sqrt{1+|\overline{\nabla}\gamma|^2},\quad h_{ij}=\frac{\phi'}{\phi w}g_{ij}-\frac{\phi}{w}\gamma_{ij}.
\end{align}
The Weingarten matrix turns into
\begin{align}\label{h3}
h_i{}^j=h_{il}g^{lj}=\frac{\phi'}{\phi w}\delta_i{}^j-\frac{\phi}{w}\gamma_{il}g^{lj}.
\end{align}
Inserting \eqref{h3} into \eqref{c1 cri} and multiplying $w^2\rho^i$ on the both sides of \eqref{c1 cri}, we obtain
\begin{align}\label{c1 cr}
\frac{w}{\phi}(\frac{\phi'}{\phi w}\delta_i{}^j-h_i{}^j)g_{lj}\gamma_l\rho^i+\b w^2|\overline{\nabla}\rho|^2=0.
\end{align}
By a direct calculation, from \eqref{g2} and \eqref{h2}, we have
\begin{align}\label{c1 gp}
g_{lj}\gamma_l=\phi\rho_j+|\overline{\nabla}\rho|^2\frac{\rho_j}{\phi}=\phi w^2\rho_j.
\end{align}
This together with \eqref{c1 cr} implies that
\begin{align}
(\b+\frac{\phi'}{\phi})|\overline{\nabla}\rho|^2=wh_i{}^j\rho_j\rho^i.
\end{align}
Since $\lbrace h_i{}^j\rbrace$ is positive-definite and $\b+\frac{\phi'}{\phi}<0$, we obtain $\overline{\nabla}\rho=0$ at the maximal point of $Q$. Thus $ Q_{\max}\leq \b\rho_{\min}<0$. And 
\begin{align*}
|\overline{\nabla} \rho|\leq \sinh(\rho
_{\max})\sqrt{e^{4\rho_{\max}\tanh(\rho_{min})}-1}.
\end{align*}
\end{proof}
\begin{rem}
When the hypersurface is uniformly convex in hyperbolic space, the gradient estimate of $\rho$ follows from the uniform upper and lower bounds of the radial function $\rho$.
\end{rem}
\begin{cor}\label{cor u}
	Based on the results of Lemma \ref{C0-est}, along the flow \eqref{flow-s}, the hypersurface $M_t$ preserves star-shaped and the support function $u$ satisfies
	\begin{align*}
	\frac{1}{C} \leq u \leq C,\quad \forall t\in[0,T).
	\end{align*}
	for some constant $C >0$, where $C$ only depends on $C_0$ estimate.
\end{cor}
\begin{proof}
	Recall \eqref{u}, 
	\begin{align*}
	u  = \frac{\phi}{\sqrt{1 +\frac{|\overline{\nabla} \rho|^2}{\phi^2}}}.
	\end{align*}
	The upper and lower bounds of $u$ follows from Lemma \ref{sec:3.1} and \ref{c1}. Besides, we have 
	\begin{align*}
	\metric{\partial_\rho}{\nu} = \frac{u}{\phi} =\frac{1}{w}= \frac{1}{\sqrt{1+ \frac{|\overline{\nabla} \rho|^2}{\phi^2}}} \geq \frac{1}{C'}
	\end{align*}
	for some $C'>0$ depends on $\max|\overline{\nabla}\rho|$ and $\min\rho$. Thus the hypersurface preserves star-shaped along the flow \eqref{s2:flow-entire}.
\end{proof}
 \begin{rem}
When the hypersurface is uniformly convex in hyperbolic space, the upper and lower bounds of the support function of $u$ follow from the uniform bound of the radial function $\rho$.
\end{rem}

\section{$C^2$ estimates}\label{sec:4}
In this section, let us assume that we have already obtained the uniform upper and lower bounds of the radial function $\rho$ and the global term $\tilde{\eta}(t)$ along the flow \eqref{s2:flow-entire}. From Lemma \ref{c1} and Corollary \ref{cor u}, the uniform bound of $\rho$ implies the upper and lower bounds of $w$ and the support function $u$. We shall establish the $C^2$ estimate under this assumption.
\subsection{The bounds of $K$}
In this subsection, we show that $K$ is bounded from above and below along \eqref{s2:flow-entire}.
\begin{lem}\label{sk low}
Along \eqref{s2:flow-entire}, there is a constant $c>0$ depending on $M_0$, $\a$, $n$ and the uniform bounds of $\rho$, $\tilde{\eta}(t)$ and $f$, such that
\begin{equation*}
K\geq c.
\end{equation*}
\end{lem}
\begin{proof}
First, we calculate the evolution equation of $\Theta=\phi^{\a}fK$
\begin{equation}\label{phi1}
\begin{aligned}
\frac{\partial}{\partial t}\Theta=&\phi(\rho)^{\a}f\partial_tK+Kf\partial_t \f \phi(\rho)^{\a}\r \\
=&\a\phi^{\a-1}\phi'\partial_t\rho fK+\phi^{\a}f\frac{\partial K}{\partial h_i{}^j}\partial_t h_i{}^j.
\end{aligned}
\end{equation}
Differentiating $\metric{V}{V}=\phi^2$, we obtain
\begin{align}
\phi \phi'\partial_t \rho=\metric{\partial_tV}{V}=\phi'\metric{\partial_t X}{V}.
\end{align}
By \eqref{s2:flow-entire}, we have
\begin{align}\label{rho ev}
\partial_t\rho=\frac{\metric{\partial_tX}{V}}{\phi}=\frac{-\Theta u+\tilde{\eta}(t)\phi^2}{\phi}=-\frac{\Theta}{w}+\tilde{\eta}(t)\phi.
\end{align} 
Inserting \eqref{evolution h} and \eqref{rho ev} into \eqref{phi1}, we get
\begin{equation}\label{phi2}
\begin{aligned}
&\frac{\partial}{\partial t}\Theta
\\
=&-\a\phi^{\a-1}\phi'\frac{\Theta}{w}fK+\a\tilde{\eta}(t)\phi'\Theta+\phi^{\a}f\frac{\partial K}{\partial h_i{}^j}\f\nabla_i\nabla^j \Theta+\Theta h_i{}^kh_k{}^j-\phi'\tilde{\eta}(t)h_i{}^j+\f\tilde{\eta}(t)u-\Theta\r\delta_i{}^j\r\\
=&\phi^{\a}f\dot{K}^{ij}\Theta_{ij}-\frac{\a\Theta^2\phi'}{w\phi}+(\a-n)\tilde{\eta}(t)\phi'\Theta+(\tilde{\eta}(t)u-\Theta)\phi^{\a}f\s_{n-1}+\Theta^2H,
\end{aligned}
\end{equation}
where we use $\frac{\partial K}{\partial h_i{}^j}\delta_i{}^j=\s_{n-1}$ and $\frac{\partial K}{\partial h_i{}^j} h_i{}^kh_k{}^j=KH$ in the last equality.
 At the spatial minimum point of $\Theta$ on $M_t$, we have $\phi^{\a}f\dot{K}^{ij}\Theta_{ij}\geq 0$. We can assume $\Theta<\inf\limits_{[0,T)} \tilde{\eta}(t)\min \frac{u}{2}$ from Corollary \ref{cor u} without loss of generality. Hence,
\begin{equation*}
\begin{aligned}
	\frac{d}{dt} \Theta_{\min} (t) \geq& -\frac{\a\Theta_{\min}^2\phi'}{w\phi}+(\a-n)\tilde{\eta}(t)\phi'\Theta_{\min}+\inf_{[0,T)} \tilde{\eta}(t)\frac{u_{\min}}{2}\phi^{\a}f\Theta_{\min}^{\frac{n-1}{n}}\\
	 \geq& -c_1 \Theta_{\min}^2(t) - c_2 \Theta_{\min}+c_3\Theta_{\min}^{\frac{n-1}{n}}
\end{aligned}
\end{equation*}
for some constant $c_1,c_2,c_3 >0$, where all of them depend on $\a$, $n$ and the uniform upper and lower bounds of $\rho$, $ \tilde{\eta}(t)$ and $f$. Hence there is a positive constant $c_4$ depending only on $c_1,c_2,c_3$, such that if $\Theta_{\min}(t)\in(0,c_4)$, we have $\frac{d}{dt} \Theta_{\min}(t) > 0$. Therefore $\Theta_{\min} (t) \geq \min \lbrace \Theta_{\min} (0), c_4, \inf \tilde{\eta}(t)\min \frac{u}{2}\rbrace$. Since $\rho$ and $f$ are bounded from above, $K= \phi^{- \a}f^{-1}\Theta$ is bounded from below by some positive constant $c$. 
\end{proof}

\begin{lem}\label{sk upp}
Along \eqref{s2:flow-entire}, there is a constant $C>0$ depending only on $M_0$, $\a$, $n$ and the uniform bounds of $\rho$, $\tilde{\eta}(t)$, $\min f$ and $\|f\|_{C^1}$, such that
\begin{equation*}
K\leq C.
\end{equation*}
\end{lem}
\begin{proof}
Let $Q= \log\Theta-\log(u-a)$, where $a = \frac{1}{2}\inf_{M \times [0,T)} u$. 
Recall \eqref{u ev},
\begin{equation}
\begin{aligned}
\frac{\partial}{\partial  t}u=&-\phi'\Theta+\phi'\tilde{\eta}(t) u+\metric{V}{\nabla \Theta }\\
=&-\phi'(\Theta-\tilde{\eta}(t)u)+\a\phi'\frac{\Theta}{\phi}\metric{V}{\nabla\rho}+\frac{\metric{V}{\nabla f}\Theta}{f}+\frac{\metric{\nabla K}{V}\Theta}{K}\\
=&-\phi'(\Theta-\tilde{\eta}(t)u)+\a\phi'|\nabla\rho|^2\Theta+\metric{V}{\nabla \log f}\Theta+\metric{\nabla \log K}{V}\Theta.
\end{aligned}
\end{equation}
Combining with \eqref{uij}, we obtain the evolution of $u$
\begin{equation}
\begin{aligned}\label{Lu} 
\partial_t u=\phi^{\a}f\dot{K}^{ij}u_{ij}-\phi'\f(n+1)\Theta-\tilde{\eta}(t)u\r+\a\phi'|\nabla\rho|^2\Theta+\frac{\metric{V}{\nabla f}}{f}\Theta+ u\phi^{\a}f\dot{K}^{ij}h_i{}^kh_{kj},
\end{aligned}
\end{equation}
where we use the Codazzi equation in the equality. 
$|\nabla f(X)|$ can be estimated as 
\begin{align}\label{fi}
|\nabla f(X)|\leq C(\min\rho, \|\rho\|_{C^1}(\mathbb{S}^n), \|f\|_{C^1}(\mathbb{S}^n)).
\end{align}
At the spatial maximum point $(p_0,t)$ of $Q$ on $M_t$, we have 
\begin{align}\label{cr-Q}
\frac{\nabla \Theta}{\Theta} = \frac{\nabla u}{u-a}.
\end{align}
Inserting \eqref{Lu} and \eqref{cr-Q} into \eqref{phi2}, we obtain
\begin{align*}
\partial_tQ=&\frac{\partial_t\Theta}{\Theta}-\frac{\partial_t u}{u-a}\\
  =&\frac{\dot{K}^{ij}\Theta_{ij}}{K}-\frac{\a\Theta\phi'}{w\phi}+(\a-n)\tilde{\eta}(t)\phi'+(\tilde{\eta}(t)u-\Theta)\frac{\s_{n-1}}{K}+\Theta H\\
  &-\frac{\phi^{\a}f\dot{K}^{ij}u_{ij}}{u-a}+\phi'\frac{(n+1)\Theta-\tilde{\eta}(t)u}{u-a}-\a\phi'|\nabla\rho|^2\frac{\Theta}{u-a}-\frac{\metric{V}{\nabla f}}{f}\frac{\Theta}{u-a}-\frac{u}{u-a}\Theta H\\
  \leq&\phi^{\a}f\dot{K}^{ij}Q_{ij}+(\a-n)\tilde{\eta}(t)\phi'+\phi'\frac{(n+1)\Theta}{u-a}+\frac{\phi\Theta|\nabla f|}{f(u-a)}-n(C_n^k)^{-\frac{1}{n}}\phi^{\a}fK^{\frac{1+n}{n}}\frac{a}{u-a},
\end{align*}
 where we assume $\Theta>\max\eta \max_{M\times[0,T)} u$ from Corollary \ref{cor u} without loss of generality. Here we also use the Newton-Maclaurin inequality $H\geq n K^{\frac{1}{n}}$ and the Cauchy-Schwartz inequality $\metric{V}{\nabla f}\leq \phi|\nabla f|$ to obtain the last inequality. Then we get
\begin{align}
\partial_t Q\leq&C_1+C_2\Theta-C_3\Theta^{\frac{n+1}{n}}
\end{align}
for some $C_1,C_2,C_3>0$ depending on $\a$, $n$, $\min f$ and $\|f\|_{C^1}$, the uniform bounds of $\rho$ and $\tilde{\eta}(t)$. Besides, there is a constant $C_4$ depending on the uniform upper and lower bounds of $u$, such that
\begin{align}
\frac{1}{C_4}e^{Q}\leq \Theta=(u-a)e^{Q}\leq C_4e^{Q}.
\end{align}
We have 
\begin{align*}
\partial_t Q\leq&C_1+C_2 C_4e^{Q}-C_3(C_4)^{-\frac{n+1}{n}}e^{\frac{n+1}{n}Q}.
\end{align*}
Therefore, $Q\leq \max\lbrace C_5,Q_{\max} (0),\max\eta \max\limits_{M\times[0,T)} u\rbrace$ where $C_5$ is a positive constant depending on $C_1$, $C_2$, $C_3$ and $C_4$. Hence $K=\phi^{-\a}f^{-1}(u-a)e^Q$ is bounded from above by some positive constant $C$ depending only on $M_0$, $\a$, $n$ $\min f$, $\|f\|_{C^1}$ and the uniform bounds of $\rho$ and $\tilde{\eta}(t)$.
\end{proof}
In the proof of Lemma \ref{sk low} and Lemma \ref{sk upp} we also obtain the uniform bounds of $\Theta$.
\begin{cor}\label{cor Phi}
	Along \eqref{s2:flow-entire}, $\Theta = \phi^{\a}f K$ is uniformly bounded, i.e.,
	\begin{align*}
	\frac{1}{C} < \Theta(x, t) < C
	\end{align*}
	for some constant $C>0$ that only depends on $M_0$, $\a$, $n$ $\min f$, $\|f\|_{C^1}$ and the uniform bounds of $\rho$ and $\tilde{\eta}(t)$.
\end{cor}

We established the $C^0$, $C^1$ estimates in Section \ref{sec:3} along \eqref{flow-s} for the proof of Theorem \ref{thm1}. Due to Lemma \ref{sk low} and Lemma \ref{sk upp}, we have the following Corollary.
\begin{cor}
Let $M_0$ be as in Theorem \ref{thm1}. When $\a>n+1$ or $\a=n+1$ with $\max f< 1$, there is a positive constant $C$ depending only on $M_0$, $\a$, $n$, $\min f$ and $\|f\|_{C^1}$ along the flow \eqref{flow-s}, such that
\begin{align}
\frac{1}{C}\leq K\leq C.
\end{align}
\end{cor}

\subsection{The bound of principal curvatures}
In this subsection, we shall show the principal curvatures of $M_t$ is uniformly bounded along the flow \eqref{s2:flow-entire} under the assumption in section \ref{sec:4}.
\begin{lem}\label{conv-c2}
	Along \eqref{s2:flow-entire}, there is a positive constant $C$ depending on $M_0$, $\a$, $n$, $\min f$, $\|f\|_{C^2}(\mathbb{S}^n)$ and the uniform bounds of $\rho$, $\tilde{\eta}$, such that the principal curvatures satisfy
	\begin{align*}
	\frac{1}{C} \leq \k_i(\cdot, t) \leq C , \quad \forall \ t \ \in [0,T) \ {\rm and} \ i=1,2,\cdots,n.
	\end{align*}
\end{lem}
\begin{proof}
 Denote by $\lambda(x,t)$ the maximal principal radii at $X(x,t)$. Let $A$ be a positive constant to be determined later. Denote by $Q'=\log \lambda(x,t) +A\rho$. Fix an arbitrary time $T_0\in (0,T)$. Assume that $Q'$ attains its maximum at $(x_0,t_0)$ provided $t\in[0,T_0)$. We then introduce a normal coordinate system $\lbrace\partial_i\rbrace$ around $(x_0,t_0)$, such that $\nabla_{\partial_i X}{}\partial_j X(x_0,t_0)=0$ for all $i,j=1,2\cdots,n$ and $h_{ij}(x_0,t_0)=\k_i(x_0,t_0)\d_{ij}$. Further, we can choose $\partial_1 |_{(x_0,t_0)}$ as the eigenvector with respect to $\lambda(x_0,t_0)$, i.e., $\lambda(x_0,t_0) = \tilde{h}^{11} (x_0, t_0)$. Assume that $\lbrace\tilde{h}^{ij}\rbrace$ is the inverse matrix of $\lbrace h_{ij}\rbrace$. Clearly
	\begin{align}\label{lambda}
	\lambda(x,t)=\max\lbrace\tilde{h}^{ij}(x,t)\xi_i\xi_j| g^{ij}(x,t)\xi_i\xi_j=1\rbrace
	\end{align}
	For the continuity, using this coordinate system we consider the auxiliary equation $Q=\log\upsilon +A\rho$, where $\upsilon=\frac{\tilde{h}^{11}}{g^{11}}$. Note that $\upsilon(x,t)\leq\lambda(x,t)$ from \eqref{lambda} and $\upsilon(x_0,t
	_0)=\lambda(x_0,t_0)$. Thus, $Q(x,t)\leq Q(x_0,t_0)$ for all $t\in[0,T_0)$. Now we can calculate the derivatives of $\upsilon$ at $(x_0,t_0)$ as follows:
	\begin{equation}\label{par t tilde h}
	\partial_t\upsilon=-(\tilde{h}^{11})^2\partial_t h_{11}+\tilde{h}^{11}\partial_t g_{11}=-\f \tilde{h}^{11}\r ^2\partial_t h_1{}^1,
	\end{equation}\label{1 deriv tilde h} 
	\begin{align}
	\nabla_i \upsilon =& \frac{\partial \tilde{h}_1{}^1}{\partial h_p{}^q} \nabla_i h_p{}^q=-\tilde{h}_1{}^p \tilde{h}_q{}^1 \nabla_i h_p{}^q,\quad \nabla_i \upsilon (x_0,t_0)=-(\tilde{h}^{11})^2 \nabla_i h_{11}.
	\end{align}
Then
\begin{align*}
	\nabla_j \nabla_i \upsilon =& \nabla_j (- \tilde{h}_1{}^p \tilde{h}_q{}^1 \nabla_i h_p{}^q)
	= -\nabla_j \tilde{h}_1{}^p \tilde{h}_q{}^1 \nabla_i h_p{}^q - \tilde{h}_1{}^p  \nabla_j \tilde{h}_q{}^1 \nabla_i h_p{}^q - \tilde{h}_1{}^p \tilde{h}_q{}^1 \nabla_j \nabla_i h_p{}^q\\
	 =& \tilde{h}_1{}^r\tilde{h}_l{}^p \tilde{h}_q{}^1 \nabla_j h_r{}^l  \nabla_i h_p{}^q + \tilde{h}_1{}^p \tilde{h}_r{}^1 \tilde{h}_q{}^s \nabla_j h_s{}^r \nabla_i h_p{}^q - \tilde{h}_1{}^p \tilde{h}_q{}^1 \nabla_j \nabla_i h_p{}^q
	 	\end{align*}
	 	and
\begin{align}\label{2 deriv tilde h} 
 \nabla_j \nabla_i \upsilon(x_0,t_0)	=& - (\tilde{h}^{11})^2 \nabla_j \nabla_i h_{11} + 2 (\tilde{h}^{11})^2 \tilde{h}^{pp} \nabla_i h_{1p} \nabla_j h_{1p}.   
\end{align}
 Here we denote by $\lbrace\tilde{h}_i{}^j\rbrace$ the inverse matrix of $\lbrace h_i{}^j\rbrace$, i.e., $\lbrace\tilde{h}_i{}^j\rbrace=\lbrace\tilde{h}^{ik}g_{kj}\rbrace$. Then We calculate the first term in \eqref{evolution h}
\begin{equation}\label{theta4}
\begin{aligned}
\nabla_j\nabla_i \Theta=&\phi^{\a}f\nabla_j\nabla_i K+K\nabla_j\nabla_i(\phi^{\a}f)+\nabla_i(\phi^{\a}f)\nabla_jK+\nabla_j(\phi^{\a}f)\nabla_iK\\
=&\phi^{\a}f\dot{K}^{pq}h_{pqij}+\phi^{\a}f\ddot{K}^{pq,rs}h_{pqi}h_{rsj}+fK\nabla_j\nabla_i(\phi^{\a})+K\phi^{\a}\nabla_j\nabla_if
+K\nabla_i\phi^{\a}\nabla_jf\\
&+K\nabla_j\phi^{\a}\nabla_if+f\nabla_i(\phi^{\a})\nabla_jK+f\nabla_j(\phi^{\a})\nabla_iK+\phi^{\a}\nabla_i f\nabla_jK+\phi^{\a}\nabla_jf\nabla_iK.
\end{aligned}
\end{equation}
Due to the Codazzi equation and Ricci identity, we have
\begin{equation}\label{kpq}
\begin{aligned}
\dot{K}^{pq}h_{pqij}=&\dot{K}^{pq}h_{piqj}=\dot{K}^{pq}\f h_{pijq}+h_{lp}R_{liqj}+h_{li}R_{lpqj}\r\\
=&\dot{K}^{pq}( h_{ijpq}+h_{lp}h_{lq}h_{ij}-h_{lp}h_{lj}h_{iq}+h_{li}h_{lq}h_{pj}-h_{li}h_{lj}h_{pq}\\
& -h_{pq}\delta_{ij}+h_{jp}\delta_{iq}-h_{qi}\delta_{pj}+h_{ij}\delta_{pq} )\\
=&\dot{K}^{pq}h_{ijpq}+KHh_{ij}-nKh_{il}h_{lj}-nK\delta_{ij}+\dot{K}^{pp}h_{ij}.
\end{aligned}
\end{equation}
It's direct to calculate
\begin{equation}\label{kpqrs}
\begin{aligned}
\ddot{K}^{pq,rs}h_{pqi}h_{rsj}=&\frac{\partial(K\tilde{h}^{pq})}{\partial h_{rs}}h_{pqi}h_{rsj}
=K\tilde{h}^{pq}\tilde{h}^{rs}h_{pqi}h_{rsj}-K\tilde{h}^{pr}\tilde{h}^{qs}h_{pqi}h_{rsj}.
\end{aligned}
\end{equation}
Hence by \eqref{theta4} at $(x_0,t_0)$,
\begin{equation}\label{theta5}
\begin{aligned}
&\nabla_j\nabla_i \Theta\\=&\phi^{\a}f\dot{K}^{pq}h_{ijpq}+\phi^{\a}fKHh_{ij}-n\phi^{\a}fKh_{il}h_{lj}-n\phi^{\a}fK\delta{ij}+\phi^{\a}f\dot{K}^{pp}h_{ij}+\phi^{\a}f\frac{\nabla_iK\nabla_jK}{K}\\
&-\phi^{\a}fK\tilde{h}^{pp}\tilde{h}^{qq}h_{pqi}h_{pqj}+\a(\a-1)\phi^{\a-2}\phi'^2fK\nabla_i\rho\nabla_j\rho +\a\phi^{\a-1}\phi'fK\nabla_j\nabla_i\rho\\
&+\phi^{\a}K\nabla_j\nabla_if
+\a\phi^{\a-1}\phi'K\nabla_i\rho\nabla_jf
+\a\phi^{\a-1}\phi'K\nabla_j\rho\nabla_if\\
&+\a\phi^{\a-1}\phi'f\nabla_i\rho\nabla_jK+\a\phi^{\a-1}\phi'f\nabla_j\rho\nabla_iK+\phi^{\a}\nabla_i f\nabla_jK+\phi^{\a}\nabla_j f\nabla_iK.
\end{aligned}
\end{equation}
Direct computation gives
\begin{equation}\label{Qij}
\begin{aligned}
\nabla_j\nabla_i Q=&\frac{\nabla_j\nabla_i\upsilon}{\upsilon}-\frac{\nabla_i\upsilon \nabla_j\upsilon}{\upsilon^2}+A\nabla_j\nabla_i\rho\\
=&- \tilde{h}^{11} \nabla_j \nabla_i h_{11} + 2 \tilde{h}^{11}\tilde{h}^{pp} \nabla_i h_{1p}\nabla_j h_{1p} -(\tilde{h}^{11})^2\nabla_i h_{11}\nabla_j h_{11} +A\nabla_j\nabla_i\rho.
\end{aligned}
\end{equation} 
Recall \eqref{evolution h} and \eqref{rho ev}. By \eqref{par t tilde h} and \eqref{theta5}, we obtain
\begin{equation}\label{Qt}
\begin{aligned}
\partial_tQ=&\frac{\partial_t\tilde{h}^{11}}{\tilde{h}^{11}}+A\partial_t\rho\\
=&-\tilde{h}^{11}\f \nabla_1\nabla^1 \Theta+\Theta h_{11}^2-\phi'\tilde{\eta}(t)h_{11}+(\tilde{\eta}(t)u-\Theta)\r-A\frac{\Theta}{w}+A\tilde{\eta}(t)\phi\\
=&-\tilde{h}^{11}[ \phi^{\a}f\dot{K}^{pq}h_{11pq}+\phi^{\a}fKHh_{11}-n\phi^{\a}fK(h_{11})^2-n\phi^{\a}fK+\phi^{\a}f\dot{K}^{pp}h_{11}\\
&+\phi^{\a}f\frac{\nabla_1K\nabla_1K}{K}
-\phi^{\a}fK\tilde{h}^{pp}\tilde{h}^{qq}(h_{pq1})^2+\a(\a-1)\phi^{\a-2}\phi'^2(\nabla_1\rho)^2fK+\a\phi^{\a-1}\phi'fK\nabla_1\nabla_1\rho\\
&+\phi^{\a}K\nabla_1\nabla_1f+2\a\phi^{\a-1}\phi'K\nabla_1\rho\nabla_1f
+2\a\phi^{\a-1}\phi'f\nabla_1\rho\nabla_1K+2\phi^{\a}\nabla_1 f\nabla_1K]\\
&-\Theta h_{11}
+\phi'\tilde{\eta}(t)-(\tilde{\eta}(t)u-\Theta)\tilde{h}^{11}-A\frac{\Theta}{w}+A\tilde{\eta}(t)\phi.
\end{aligned}
\end{equation} 
Substituting \eqref{Qij} into \eqref{Qt}, we get
\begin{equation}\label{Qt2}
\begin{aligned}
\partial_tQ=&\phi^{\a}f\dot{K}^{ij}\nabla_j\nabla_iQ-2\phi^{\a}fK\tilde{h}^{11}\tilde{h}^{ij}\tilde{h}^{pp} \nabla_i h_{1p}\nabla_j h_{1p}+\phi^{\a}fK\tilde{h}^{ij}(\tilde{h}^{11})^2\nabla_i h_{11}\nabla_j h_{11}\\
&-A\phi^{\a}f\dot{K}^{ij}\nabla_j\nabla_i\rho
-\phi^{\a}fKH+n\phi^{\a}fKh_{11}+n\phi^{\a}fK\tilde{h}^{11}-\phi^{\a}fK\tilde{h}^{pp}
-\phi^{\a}f\frac{(\nabla_1K)^2}{K}\tilde{h}^{11}\\
&+\phi^{\a}fK\tilde{h}^{11}\tilde{h}^{pp}\tilde{h}^{qq}(h_{pq1})^2
-\a(\a-1)\phi^{\a-2}\phi'^2(\nabla_1\rho)^2fK\tilde{h}^{11}-\a\phi^{\a-1}\phi'fK\tilde{h}^{11}\nabla_1\nabla_1\rho \\
&-\phi^{\a}K\tilde{h}^{11}\nabla_1\nabla_1f
-2\a\phi^{\a-1}\phi'K\tilde{h}^{11}\nabla_1\rho\nabla_1f
-2\a\phi^{\a-1}\phi'f\tilde{h}^{11}\nabla_1\rho\nabla_1K\\
&-2\phi^{\a}\tilde{h}^{11}\nabla_1 f\nabla_1K-\Theta h_{11}+\phi'\tilde{\eta}(t)-(\tilde{\eta}(t)u-\Theta)\tilde{h}^{11}-A\frac{\Theta}{w}+A\tilde{\eta}(t)\phi.
\end{aligned}
\end{equation}
Dividing \eqref{Qt2} by $\Theta$ on both sides, we obtain at $(x_0,t_0)$
\begin{equation}\label{Qt3}
\begin{aligned}
\frac{\partial_tQ}{\Theta}\leq&\tilde{h}^{ij}\nabla_j\nabla_iQ-A\tilde{h}^{ij}\nabla_j\nabla_i\rho+nh_{11}+n\tilde{h}^{11}-(\nabla_1 \log K)^2\tilde{h}^{11}
-\a(\a-1)(\frac{\phi'}{\phi})^2(\nabla_1\rho)^2\tilde{h}^{11}\\
&-\a\frac{\phi'}{\phi}\tilde{h}^{11}\nabla_1\nabla_1\rho
-\frac{\nabla_1\nabla_1 f}{f}\tilde{h}^{11}
-2\a\frac{\phi'}{\phi}\tilde{h}^{11}\nabla_1\rho\nabla_1\log f-2\a\frac{\phi'}{\phi}\tilde{h}^{11}\nabla_1\rho\nabla_1\log K\\
&-2\tilde{h}^{11}\nabla_1 \log f\nabla_1\log K+\frac{\phi'}{\Theta}\tilde{\eta}(t)+\tilde{h}^{11}+A\frac{\tilde{\eta}(t)\phi}{\Theta},
\end{aligned}
\end{equation}
where we wipe off some nonpositive terms. 
Recall \eqref{rhoij}, we have at $(x_0,t_0)$
\begin{align}\label{rho e}
|\nabla\rho|^2=\frac{\sum_{i=1}^n\metric{X_i}{V}^2}{\phi^2}=\frac{|V|^2}{\phi^2}-\frac{\metric{\nu}{V}^2}{\phi^2}\leq 1-(\frac{u}{\phi})^2<1-c_1 
\end{align}
where $c_1>0$ depends on the lower bound of $u$ and the upper bound of $\rho$.
Now substituting \eqref{rhoij} into \eqref{Qt3} and using the Cauchy-Schwartz inequality, we obtain at $(x_0,t_0)$
\begin{equation}\label{Qt5}
\begin{aligned}
\frac{\partial_tQ}{\Theta}\leq&\tilde{h}^{ij}\nabla_j\nabla_iQ-A\tilde{h}^{ii}\f\frac{\phi'}{\phi}\f 1-(\nabla_i\rho)^2\r-\frac{u h_{ii}}{\phi}\r+n h_{11}+(n+1)\tilde{h}^{11}\\
&-\a(\a-1)(\frac{\phi'}{\phi})^2(\nabla_1\rho)^2\tilde{h}^{11}
-\a\frac{\phi'}{\phi}\tilde{h}^{11}\f \frac{\phi'}{\phi} \f 1-(\nabla_1\rho)^2 \r-\frac{u h_{11}}{\phi}\r
-\frac{\nabla_1\nabla_1 f}{f}\tilde{h}^{11}\\
&-2\a\frac{\phi'}{\phi}\tilde{h}^{11}\nabla_1\rho\nabla_1\log f
+2\a^2(\frac{\phi'}{\phi})^2(\nabla_1\rho)^2\tilde{h}^{11}
+2(\nabla_1 \log f)^2\tilde{h}^{11}+\frac{\phi'}{\Theta}\tilde{\eta}(t)+A\frac{\tilde{\eta}(t)\phi}{\Theta}.
\end{aligned}
\end{equation}
Using the maximum principle, by \eqref{rho e} at $(x_0,t_0)$, we have
\begin{equation}\label{Qt4}
\begin{aligned}
\frac{\partial_tQ}{\Theta}\leq&-Ac_1\tilde{h}^{11}\frac{\phi'}{\phi}+A(\frac{nu }{\phi}+\frac{\tilde{\eta}(t)\phi}{\Theta})+nh_{11}+(n+1)\tilde{h}^{11}
-\a(\a-1)(\frac{\phi'}{\phi})^2(\nabla_1\rho)^2\tilde{h}^{11}\\
&-\a(\frac{\phi'}{\phi})^2(1-\rho_1^2)\tilde{h}^{11}+\a\frac{u\phi'}{\phi^2}-\frac{\nabla_1\nabla_1 f}{f}\tilde{h}^{11}-2\a\frac{\phi'}{\phi}\tilde{h}^{11}\nabla_1\rho\nabla_1\log f\\
&+2\a^2(\frac{\phi'}{\phi})^2(\nabla_1\rho)^2\tilde{h}^{11}+2(\nabla_1 \log f)^2\tilde{h}^{11}+\frac{\phi'}{\Theta}\tilde{\eta}(t).
\end{aligned}
\end{equation}
Write $X\in\mathbb{H}^{n+1}$ as $X=(\rho,\theta)\in \mathbb{R}^+\times\mathbb{S}^n$. Then we extend $f$ to $\f \mathbb{H}^{n+1},g_{\mathbb{H}^{n+1}} \r$ as
 \begin{align}\label{f def}
 f(X)=f(\theta).
 \end{align}
We have by Reilly formula,
\begin{align}\label{hess f}
D_jD_i f=\nabla_j\nabla_i f+h_{ij}D_{\nu} f. 
\end{align}
where $D$ is the standard Levi-Civita connection with respect to $g_{\mathbb{H}^{n+1}}$. 
Then at $(x_0,t_0)$, $\nabla_1\nabla_1f$ can be further estimated as 
\begin{align}
|\nabla_1\nabla_1f(x_0,t_0)|\leq C\f\max \rho,\min \rho,\|f\|_{C^2}(\mathbb{S}^n)\r(1+h_{11}(x_0,t_0)).
\end{align}
By \eqref{fi} and \eqref{hess f}, We obtain at $(x_0,t_0)$   
\begin{align}\label{c2 end}
0\leq\frac{\partial_tQ}{\Theta}\leq-\f\frac{Ac_1}{2}-c_2\r\tilde{h}^{11}-A\f\frac{c_1\tilde{h}^{11}}{2}-c_3\r+c_4+c_5 \frac{1}{\tilde{h}^{11}}
\end{align}
for some positive constants $c_1$, $c_2$, $c_3$, $c_4$, $c_5$ depends on the uniform upper and lower bounds of the $\rho$, $\tilde{\eta}$, $u$ from Corollary \ref{cor u}, $K$ from Lemma \ref{sk low}-\ref{sk upp} and $\a$, $n$, $\inf f$, $\|f\|_{C^2}(\mathbb{S}^n)$. We deduce from \eqref{c2 end} that if we choose $A=\frac{2c_2}{c_1}$, $\lambda(x_0,t_0)$ can't be too large, i.e., $\lambda(x_0,t_0)$ has a uniform upper bound $C(c_i,i=1,\cdots,5)$ independent of time $T_0$. Hence we have the uniform upper bound of the principal raddi of $M_t$, which means that the principal curvatures are bounded from below by some positive constant $c_6$. Meanwhile, by Lemma \ref{sk upp}, one can get
\begin{align*}
C\geq K\geq c_6^{n-1}\kappa_{\max}
\end{align*}
for some constant $C$ from Lemma \ref{sk upp}. 
 Hence, the principal curvatures are bounded from above. This completes the proof of Lemma \ref{conv-c2}.
\end{proof}

Now we have obtained the a priori estimates of the flow \eqref{flow-s}. By Lemma \ref{sec:3.1}, Lemma \ref{c1} and Lemma \ref{sk low}, these flows have short time existence. Using the $C^2$ estimate given in Lemma \ref{conv-c2}, due to \cite[Theorem 6]{And04}, we obtain the $C^{2,\lambda}$ estimate of the scalar equation \eqref{s1:flow-rn}. 
Then the standard regularity theory implies the estimates for higher order derivatives. Hence, we obtain the long time existence and regularity for these flows.
\begin{thm}\label{longexsits}
	Let $M_0$ be be as in Theorem \ref{thm1}. Assume that $f$ is a smooth positive function on $\mathbb{S}^{n}$. If $\a>n+1$ or $\a=n+1$ and $f$ satisfies $f<1$, then the smooth uniformly convex solution to the flow \eqref{flow-s} exists for all time $t \in [0, +\infty)$ and there is a constant $C_{m,\lambda}>0$ depending on $M_0$, $\a$, $n$, $m$, $\lambda$, and $f$ such that 
	\begin{align*}
	\|\rho\|_{C^{m,\lambda}_{(\theta,t)}(\mathbb{S}^n\times[0,+\infty))}\leq C_{m,\lambda}.
	\end{align*}
\end{thm}
\begin{rem}
If the Gauss curvature $K$ in flow \eqref{flow-s} is replaced by the $k$-th mean curvature $\s_k$ of the hypersurface, using the similar argument we can obtain the same long time existence results as in Theorem \ref{longexsits}.
\end{rem}
In other words, there is a subsequence of time $\lbrace t_i\rbrace$, such that $\lbrace M_{t_i}\rbrace$ converge to a smooth positive uniformly convex hypersurface $M_{\infty}$  in $C^{\infty}$ topology. Next we shall see $M_{\infty}$ is a solution to \eqref{alex2}. 
\section{Proof of Theorem \ref{thm1}}\label{sec:5}
In this section, we consider the hyperbolic space as the hyperboloid in Lorentz space $\mathbb{R}^{n+1,1}$, where
\begin{align*}
\mathbb{H}^{n+1}=\lbrace(x_1,\cdots,x_{n+1},x_{n+2})\in\mathbb{R}^{n+1,1}|\sum_{i=1}^{n+1}x_i^2-x_{n+2}^2=-1, \ x_{n+2}>0\rbrace.
\end{align*}

Let $p=(0,\cdots,0,1)$ and $L_p$ be the tangent plane of $\mathbb{H}^{n+1}$ at $p$. Denote the projection from $\mathbb{H}^{n+1}$ to $ L_p$ as

\begin{align*}
\pi_p:\mathbb{H}^{n+1}&\rightarrow L_p\\
z=(x_1,\cdots,x_{n+1},x_{n+2})&\mapsto \frac{z}{-\metric{z}{p}}=(\frac{x_1}{x_{n+2}},\cdots,\frac{x_{n+1}}{x_{n+2}},1).
\end{align*}
Since $\sum_{i=1}^{n+1}\f \frac{x_i}{x_{n+2}}\r^2<1$ on $\pi_p(\mathbb{H}^{n+1})$, $\pi_p(\mathbb{H}^{n+1})$ is contained in the unit ball $B^{n+1}_p(1)$ centered at $p$. Recall that $\rho(q)$ is the geodesic distance between $q$ and $p$ in $\mathbb{H}^{n+1}$. If we regard $r$ as the Euclidean distance from $\pi(q)$ to $p$ in $B^{n+1}_p(1)$, we have the following relation
\begin{align}\label{proj re}
r(\theta)=\tanh\rho(\theta).
\end{align}
Direct computation gives
\begin{align}\label{cosh}
\frac{1}{(\cosh\rho)^2}=1-r^2,
\end{align}
\begin{align}\label{sinh}
\sinh\rho=\frac{r}{\sqrt{1-r^2}},
\end{align}

\begin{align}
|\overline{\nabla}\rho|^2=\frac{|\overline{\nabla} r|^2}{(1-r^2)^2},
\end{align}
\begin{align}\label{w np}
w=\sqrt{1+\frac{|\overline{\nabla} \rho|^2}{\phi^2}}=\frac{r}{\hat{u}}\sqrt{\frac{1-\hat{u}^2}{1-r^2}},
\end{align}
\begin{align}\label{u np}
u=\frac{\hat{u}}{\sqrt{1-\hat{u}^2}}.
\end{align}
Here we can see $u>0$ as long as $1>\hat{u}>0$, which means that $\pi_p(M_t)$ remains star-shaped as long as $M_t$ does. From \eqref{proj re}, we have the scalar equation of the radial graph of $\pi_p(M_t)$ under the flow \eqref{s2:flow-entire}($\tilde{\eta}\equiv 1$),
\begin{align}\label{r eu}
\partial_t r(\t,t)=\partial_{\rho}(\tanh\rho)\partial_t\rho=\frac{-w \phi(\rho)^{\a}f(\theta)K}{(\cosh\rho)^2}+\frac{\phi}{(\cosh\rho)^2}
\end{align}
where we denote by $(r(\t),\theta)$ the radial graph of $\pi_p(M_t)$. Similar to \eqref{u}-\eqref{hi_j}, we have in $\hat{M}_t=\pi(M_t)$,

\begin{align}
\hat{\nu}=\frac{r}{\sqrt{r^2+|\overline{\nabla} r|^2}}(\partial_r-\frac{\overline{\nabla} r}{r^2}),
\end{align}
\begin{align}
\hat{u}=\frac{r^2}{\sqrt{r^2+|\overline{\nabla} r|^2}},
\end{align}
\begin{align}
\hat{g}_{ij}=r_ir_j+r^2\delta_{ij}
\end{align}
and 
\begin{align}
\hat{h}_{ij}=\frac{-rr_{ij}+2r_ir_j+r^2\delta_{ij}}{\sqrt{r^2+|\overline{\nabla} r|^2}}.
\end{align}
Therefore,
\begin{align}\label{hatk1}
\hat{K}=\frac{\det{\hat{h}_{ij}}}{\det{\hat{g}_{ij}}}=\frac{\det(-rr_{ij}+2r_ir_j+r^2\delta_{ij})}{(r^2+|\overline{\nabla} r|^2)^{\frac{n+2}{2}}r^{2n-2}}.
\end{align}
Recall \eqref{gij} and \eqref{hi_j}. Substituting \eqref{proj re} in \eqref{hatk1}, we have(see \cite{CH21})
\begin{align}\label{hat k}
K=\hat{K}\f\frac{(r^2+|\overline{\nabla} r|^2)}{r^2+\frac{|\overline{\nabla} r|^2}{1-r^2}}\r^{\frac{n+2}{2}}=\hat{K}\f\frac{1-r^2}{1-\hat{u}^2}\r^{\frac{n+2}{2}}.
\end{align}

Now we calculate the scalar parabolic PDE of the support function $\hat{u}$ of $\pi_p(M_t)$. Substituting \eqref{cosh}, \eqref{sinh}, \eqref{w np} and \eqref{hat k} into \eqref{r eu}, we have
\begin{equation}\label{u eu}
\begin{aligned}
\partial_t \hat{u}(x,t)=\frac{\hat{u}}{r}\partial_t r(\theta,t)&=(- \phi(\rho)^{\a}f(\theta)Kw+\phi)(1-r^2)\frac{\hat{u}}{r}\\
&=-r^{\a}f(\theta)\hat{K}(1-r^2)^{\frac{n+3-\a}{2}}({1-\hat{u}^2})^{-\frac{n+1}{2}}+\hat{u}\sqrt{1-r^2}.
\end{aligned}
\end{equation}
Denote by $\psi(r)=r^{\a}(1-r^2)^{\frac{n+2-\a}{2}}$ and $\varphi(\hat{u})=\hat{u}^{-1}({1-\hat{u}^2})^{-\frac{n+1}{2}}$. \eqref{u eu} becomes 
\begin{align}
\partial_t \hat{u}=-\psi(r)\sqrt{1-r^2}f(\theta)\varphi(\hat{u})\hat{u}\hat{K}+\hat{u}\sqrt{1-r^2}.
\end{align}
Let $\Psi=\displaystyle\int_a^{r}\psi^{-1}(s)s^n \mathrm{d}s$, $\Omega=\displaystyle\int_a^{\hat{u}} \varphi(s) \mathrm{d}s$ and \begin{equation}\label{Q}
 Q(t)=\displaystyle\int_{\mathbb{S}^n} \Psi f^{-1} \mathrm{d}\theta_{\mathbb{S}^n}-\displaystyle\int_{\mathbb{S}^n}\Omega \mathrm{d}\sigma_{\mathbb{S}^n}.   
\end{equation} 
where $0<a<1$ some constant to be chosen later for convenience. $a=\frac{1}{2}\min\limits_{\hat{M_0}}r=\frac{1}{2}\min\limits_{\hat{M_0}}u$. We have the following properties of $Q$ along \eqref{u eu} (referring to \cites{LL20, LSW20a}).

\begin{lem}\label{Q mon}
Along \eqref{u eu}, $Q(t)$ is non-decreasing and the equality holds if and only if $\pi_p(M_t)$ satisfies the following equation
\begin{equation}\label{ell eu}
r^{\a}(1-r^2)^{\frac{n+2-\a}{2}} f\hat{K}=\hat{u} (1-\hat{u}^2)^{\frac{n+1}{2}}.
\end{equation} 

\end{lem}

\begin{proof}
\begin{equation}\label{Q1}
\begin{aligned}
Q'(t)=&\partial_t\displaystyle\int_{\mathbb{S}^n} \Psi f^{-1} \mathrm{d}\theta_{\mathbb{S}^n}-\partial_t\displaystyle\int_{\mathbb{S}^n}\Omega \mathrm{d}\sigma_{\mathbb{S}^n}\\
=&\displaystyle\int_{\mathbb{S}^n}\psi^{-1}(r)r^n\partial_t r f^{-1} \mathrm{d}\theta_{\mathbb{S}^n}-\displaystyle\int_{\mathbb{S}^n}\varphi(\hat{u})\partial_t \hat{u} \mathrm{d}\sigma_{\mathbb{S}^n}.
\end{aligned}
\end{equation}
Since $\frac{\partial_t r(\theta,t)}{r}=\frac{\partial_t \hat{u}(x,t)}{\hat{u}}$ and $r^{n+1}\mathrm{d}\theta_{\mathbb{S}^n}=\frac{\hat{u}}{\hat{K}} \mathrm{d}\sigma_{\mathbb{S}^n}$, \eqref{Q1} becomes 
\begin{equation}
\begin{aligned}
Q'(t)=&\displaystyle\int_{\mathbb{S}^n}\f \psi^{-1}(r)f^{-1}(\theta)\hat{K}^{-1}-\varphi(\hat{u})\r\partial_t \hat{u} \mathrm{d}\sigma_{\mathbb{S}^n}\\
=&\displaystyle\int_{\mathbb{S}^n}\f \psi^{-1}(r)f^{-1}(\theta)\hat{K}^{-1}-\varphi(\hat{u})\r^2\psi(r)\sqrt{1-r^2}f(\theta)\hat{u}\hat{K}\mathrm{d}\sigma_{\mathbb{S}^n}\\
\geq&0.
\end{aligned}
\end{equation}
Here the equality holds if and only if $\varphi(r)\psi(\hat{u})f(\theta)\hat{K}\equiv 1$.
\end{proof}
By Lemma \ref{sec:3.1}, we obtain the uniform upper and lower bounds of $\rho$ in the hyperbolic space. Thus from \eqref{proj re}, there exists $0<c_1<c_2<1$ independent of time, such that $c_1<r<c_2$ and $c_1<\hat{u}<c_2$. From \eqref{Q}, $Q(t)$ is uniformly bounded independent of time, i.e., there exists a constant $C>0$, such that $|Q(t)|\leq C$. 
Since
\begin{equation}
\begin{aligned}
Q(t)-Q(0)
=\displaystyle\int_0^t\displaystyle\int_{\mathbb{S}^n}\f \psi^{-1}(r)f^{-1}(\theta)\hat{K}^{-1}-\varphi(\hat{u})\r^2\psi(r)\sqrt{1-r^2}f(\theta)\hat{u}\hat{K}\mathrm{d}\sigma_{\mathbb{S}^n}dt,
\end{aligned}
\end{equation}
we obtain
\begin{align}
\displaystyle\int_0^{\infty}\displaystyle\int_{\mathbb{S}^n}\f \psi^{-1}(r)f^{-1}(\theta)\hat{K}^{-1}-\varphi(\hat{u})\r^2\psi(r)\sqrt{1-r^2}f(\theta)\hat{u}\hat{K}\mathrm{d}\sigma_{\mathbb{S}^n}dt<\infty.
\end{align}
By Theorem \ref{longexsits}, there exists a subsequence $\lbrace t_i\rbrace$, such that $M_{t_i}$ converge smoothly to $M_{\infty}$ and
\begin{align*}
\displaystyle\int_{\mathbb{S}^n}\f \psi^{-1}(r)f^{-1}(\theta)\hat{K}^{-1}-\varphi(\hat{u})\r^2\mathrm{d}\sigma_{\mathbb{S}^n}\rightarrow 0
\end{align*}
as $t_i\rightarrow\infty$, since $\hat{K}$ is uniformly bounded from below by \eqref{hat k} and Lemma \ref{sk low}.
Thus $\hat{M}_{\infty}$ satisfies
\begin{align}
 \psi(r)f(\theta)\hat{K}=\varphi(\hat{u})^{-1}.
\end{align}
Combining with \eqref{proj re}, \eqref{cosh}, \eqref{sinh} etc., we obtain $M_{\infty}$ satisfies \eqref{alex2}.

Now we show the uniqueness of the solution of \eqref{alex2} when $\a\geq n+1$. First we assume that there exists two solutions $\rho_1$ and $\rho_2$. Let $\gamma(\rho)=\log(1-\frac{2}{e^{\rho}+1})$ and $G=\gamma_1-\gamma_2$, then $\overline{\nabla}\gamma=\frac{\overline{\nabla} \rho}{\phi}$. Assume $G$ attains its maximal point at $\theta_0$ and $G(\theta_0)>0$, which implies $\rho_1(\theta_0)>\rho_2(\theta_0)$. At $\theta_0$, we have 
\begin{align}\label{cri uq}
\overline{\nabla}\gamma_1=\overline{\nabla}\gamma_2
\end{align}
and $\overline{\nabla}^2_{ij}G\leq 0$, i.e., 	
\begin{align}\label{eta2}
\overline{\nabla}^2_{ij}\gamma_1\leq\overline{\nabla}^2_{ij}\gamma_2.
\end{align}
Substituting $\gamma$ into \eqref{gij} and \eqref{hij}, we have  
\begin{align}\label{gij eta}
g_{ij}=\phi^2\f \delta^{ij}+\gamma_i\gamma_j\r, \quad g^{kj}=\frac{1}{\phi^2}\f \delta^{kj}-\frac{\gamma_k\gamma_j}{1+|\overline{\nabla}\gamma|^2}\r
\end{align}
and 
\begin{align}\label{hij eta}
h_{ik}=\frac{\phi}{\sqrt{1+|\overline{\nabla}\gamma|^2}}\f-\gamma_{ik}+\phi'\gamma_i\gamma_k+\phi'\delta_{ik}\r.
\end{align}
Plugging \eqref{gij eta} and \eqref{hij eta} in  \eqref{alex2} we have
\begin{align}
\sigma_n(h_i{}^j)=\frac{f}{\phi(\rho)^{\a-1}\sqrt{1+|\overline{\nabla}\gamma|^2}},
\end{align}
where
\begin{align}
h_i{}^j=\frac{1}{\phi\sqrt{1+|\overline{\nabla}\gamma|^2}}\f-\gamma_{ik}+\phi'\gamma_i\gamma_k+\phi'\delta_{ik}\r\f \delta^{kj}-\frac{\gamma_k\gamma_j}{1+|\overline{\nabla}\gamma|^2}\r.
\end{align}
Then 
\begin{align}
\sigma_n\f \frac{1}{\sqrt{1+|\overline{\nabla}\gamma|^2}}\f-\gamma_{ik}+\phi'\gamma_i\gamma_k+\phi'\delta_{ik}\r\f \delta^{kj}-\frac{\gamma_k\gamma_j}{1+|\overline{\nabla}\gamma|^2}\r \r=\frac{f}{\phi(\rho)^{\a-n-1}\sqrt{1+|\overline{\nabla}\gamma|^2}}.
\end{align}
Using \eqref{cri uq}, we have at $\theta$
\begin{align}\label{eq}
&\phi(\rho_1)^{\a-n-1}\sigma_n\f -(\gamma_1){}_{ik}+\phi(\rho_1)'(\gamma_1){}_{i}(\gamma_1){}_{k}+\phi(\rho_1)'\delta_{ik} \r\nonumber\\
=&\phi(\rho_2)^{\a-n-1}\sigma_n\f -(\gamma_2){}_{ik}+\phi(\rho_2)'(\gamma_2){}_{i}(\gamma_2){}_{k}+\phi(\rho_2)'\delta_{ik} \r.
\end{align}
Since $\rho_1(\theta)>\rho_2(\theta)$ and both of the solutions are uniformly convex, by \eqref{cri uq} and \eqref{eta2}
\begin{align}
-(\gamma_1){}_{ik}+\phi'(\rho_1)(\gamma_1){}_{i}(\gamma_1){}_{k}+\phi'(\rho_1)\delta_{ik}>-(\gamma_2){}_{ik}+\phi(\rho_2)'(\gamma_2){}_{i}(\gamma_2){}_{k}+\phi(\rho_2)'\delta_{ik}>0.
\end{align}
When $\a\geq n+1$, we have
\begin{align*}
&\phi(\rho_1)^{\a-n-1}\sigma_n\f -(\gamma_1){}_{ik}+\phi(\rho_1)'(\gamma_1){}_{i}(\gamma_1){}_{k}+\phi(\rho_1)'\delta_{ik} \r\\
>&\phi(\rho_2)^{\a-n-1}\sigma_n\f -(\gamma_2){}_{ik}+\phi(\rho_2)'(\gamma_2){}_{i}(\gamma_2){}_{k}+\phi(\rho_2)'\delta_{ik} \r
\end{align*}
which is contrary to \eqref{eq}. So $G\leq 0$, which implies $\gamma_1\equiv\gamma_2$ and $\rho_1\equiv\rho_2$.
Now we complete the proof of Theorem \ref{thm1}.

\section{Proof of Theorem \ref{thm2}}\label{sec:6}
In this section, we consider the flow \eqref{flow-s} for $\a=n+1$ and the corresponding even Alexandrov problem in $\mathbb{H}^{n+1}$. Throughout this section, if not specified, we regard $c_i$, $C_i$ for $i\in\mathbb{N}$ as some positive constants. 
By \eqref{s1:flow-rn}, when $\a=n+1$ the scalar equation of the radial function $\rho$ becomes 
\begin{equation}\label{flow-r n+1}
	\left\{\begin{aligned}
	\frac{\partial}{\partial t}\rho(\theta,t)=&-\f\phi(\rho)\r^{n+1}f(\theta)wK +\phi(\theta,t),\quad \text{for }(\theta,t)\in\mathbb{S}^n\times[0,+\infty),\\
	\rho(\cdot,0)=& \rho_0(\cdot).
	\end{aligned}\right.
	\end{equation}
By the projection $\pi_p$ in Section \ref{sec:5} and the relation between $M_t$ and $\pi_p(M_t)$, \eqref{proj re}-\eqref{u np} and \eqref{hat k}, we obtain the scalar parabolic PDE of the radial function $r$ of $\pi_p(M_t)$ along the flow \eqref{flow-s}
\begin{align}\label{flow s n+1}
\partial_t r=-r^{n+2}(1-r^2)\hat{u}^{-1}(1-\hat{u}^2)^{-\frac{n+1}{2}}f\hat{K}+r\sqrt{1-r^2}.
\end{align}

Then we have the following results.
\begin{lem}\label{thm2 c0}
Let $\rho$ be a smooth, positive, uniformly convex and origin-symmetric solution to \eqref{flow-r n+1} on $\mathbb{S}^n\times[0,T)$. If $f$ is a smooth positive even function on $\mathbb{S}^n$ satisfying $\int_{\mathbb{S}^n} f^{-1}\mathrm{d}\theta_{\mathbb{S}^n}>|S^n|$, then there exists a positive constant $C$ depending on $n$, $\max f$, $\min f$ and the initial hypersurface, such that
\begin{align}
\frac{1}{C}\leq \rho\leq C,\quad \forall t\in[0,T).
\end{align}
\end{lem}
\begin{proof}
The upper bound of $\rho$ is obtained directly from the proof of Lemma \ref{C0-est}, i.e., there exists a positive constant $c_1$, such that $\rho\leq c_1$. Combining with \eqref{sinh}, we obtain that the radial function $r$ of $\pi_p(M_t)$ is away from $1$, i.e., there exists a positive constant $c_2$, such that $\hat{u}\leq r\leq c_2<1$. Under the flow \eqref{flow s n+1}, we have the following monotone non-decreasing function by Lemma \ref{Q mon}
\begin{align}
Q(t)=\displaystyle\int_{\mathbb{S}^n}\int_a^r s^{-1}(1-s^2)^{-\frac{1}{2}} f^{-1} \mathrm{d}s \mathrm{d}\theta_{\mathbb{S}^n}-\displaystyle\int_{\mathbb{S}^n}\int_a^{\hat{u}}s^{-1}(1-s^2)^{-\frac{n+1}{2}}\mathrm{d}s \mathrm{d}\sigma_{\mathbb{S}^n}.
\end{align}
where we can choose $a=c_2$ without loss of generality. Then there exists a positive constant $C_1$ depending on the initial hypersurface of \eqref{flow-s}, such that
\begin{equation}\label{thm2 Q}
\begin{aligned}
-C_1\leq Q(t)\leq &-f_{\max}^{-1}\displaystyle\int_{\mathbb{S}^n}\int_r^{c_2} s^{-1}\mathrm{d}s \mathrm{d}\theta_{\mathbb{S}^n}
+(1-c_2^2)^{-\frac{n+1}{2}}\displaystyle\int_{\mathbb{S}^n}\int_{\hat{u}}^{c_2} s^{-1}\mathrm{d}s \mathrm{d}\sigma_{\mathbb{S}^n}\\
= & -f_{\max}^{-1}\displaystyle\int_{\mathbb{S}^n}(\log c_2-\log r) \mathrm{d}\theta_{\mathbb{S}^n}+(1-c_2^2)^{-\frac{n+1}{2}}\displaystyle\int_{\mathbb{S}^n}(\log c_2-\log \hat{u}) \mathrm{d}\sigma_{\mathbb{S}^n}\\
\leq & C_2+f_{\max}^{-1}\displaystyle\int_{\mathbb{S}^n}\log r \ \mathrm{d}\theta_{\mathbb{S}^n}-(1-c_2^2)^{-\frac{n+1}{2}}\displaystyle\int_{\mathbb{S}^n}\log \hat{u} \  \mathrm{d}\sigma_{\mathbb{S}^n}.
\end{aligned}
\end{equation}
Assume $r_{\min}(t)$ is attained at $\theta_0=(1,\vec{0})$ at any fixed time $t$. Here we define by $\vec{0}$ an $n$-dimensional zero vector. We parametrize any point $\theta\in\mathbb{S}^n$ as 
\begin{align}\label{theta para}
\theta=(\cos \theta_1,\sin \theta_1\vec{x}),
\end{align}
where $0\leq\theta_1\leq \pi$, and $\vec{x}=(x_2,\cdots,x_{n+1})\in \mathbb{S}^{n-1}$ is an $n$-dimensional unit vector. We have $\metric{\theta}{\theta_0}=\cos \theta_1$. Then $r(\theta,t)\leq \frac{r_{\min}(t)}{|\cos \theta_1|}$ because the flow hypersurface $M_t$ is strictly convex and origin-symmetric.
Assume $\hat{u}_{\max}(t)$ is attained at $v_0\in \mathbb{S}^n$. We have $\hat{u}(v,t) \geq \hat{u}_{\max}(t)|\metric{v}{v_0}|$ (referring to \cite[p.44]{Sch13}). Take the direction of $v_0$ as $x$-axis and regard $v_1$ as the angle with the $x$-axis. We parametrize any point $v$ in $\mathbb{S}^n$ as 
\begin{align}\label{nu para}
v=\f\cos v_1,\sin v_1\vec{y}\r,
\end{align}
where $0\leq v_1\leq \pi$, and $\vec{y}=(y_2,\cdots,y_{n+1})\in \mathbb{S}^{n-1}$ is an $n$-dimensional unit vector. We have $\metric{v}{v_0}=\cos v_1$. By a direct calculation, the area measure $\mathrm{d}\sigma_{\mathbb{S}^n}$ becomes
\begin{align}\label{area measure v}
\mathrm{d}\sigma_{\mathbb{S}^n}=(\sin v_1)^{n-1} \mathrm{d}v_1\mathrm{d}\sigma_{\mathbb{S}^{n-1}}.
\end{align}
Similarly,
\begin{align}\label{area measure t}
\mathrm{d}\t_{\mathbb{S}^n}=(\sin \t_1)^{n-1} \mathrm{d}\t_1\mathrm{d}\t_{\mathbb{S}^{n-1}}.
\end{align}
Then
\begin{equation}\label{thm2 logv}
\begin{aligned}
\displaystyle\int\limits_{\mathbb{S}^n}\log \hat{u} \mathrm{d}\sigma_{\mathbb{S}^n}
\geq & \int\limits_{\mathbb{S}^n}\log(\hat{u}_{\max}|\cos v_1|) \mathrm{d}\sigma_{\mathbb{S}^n}\\
=&2\int_0^{\frac{\pi}{2}}\int_{\mathbb{S}^{n-1}}\f \log\hat{u}_{\max}+\log\cos v_1\r (\sin v_1)^{n-1} \mathrm{d}v_1\mathrm{d}\sigma_{\mathbb{S}^{n-1}}\\
\geq& |S^n|\log\hat{u}_{\max}+2|S^{n-1}|\int_0^{\frac{\pi}{2}}\log\cos v_1\mathrm{d}v_1
\geq|S^n|\log\hat{u}_{\max}-C_3.
\end{aligned}
\end{equation} 
The second term in the last inequality is convergent, since $\log (\cos v_1)\geq\log \f \frac{\pi}{4}-\frac{v_1}{2}\r$ at $v_1\in[0,\frac{\pi}{2}]$.
Besides, we have 
\begin{equation}\label{thm2 logt}
\begin{aligned}
\displaystyle\int\limits_{\mathbb{S}^n}\log r \mathrm{d}\theta_{\mathbb{S}^n}
\leq &\int\limits_{\mathbb{S}^n}\log \frac{r_{\min}}{|\cos \theta_1|} \mathrm{d}\theta_{\mathbb{S}^n}\\
=&2\int_0^{\frac{\pi}{2}}\int_{\mathbb{S}^{n-1}}\f\log r_{\min}-\log \cos \theta_1\r (\sin \theta_1)^{n-1}\mathrm{d}\theta_1\mathrm{d}\theta_{\mathbb{S}^{n-1}}\\
\leq&|S^n|\log r_{\min}-2|S^{n-1}|\int_0^{\frac{\pi}{2}}\log\cos \theta_1 \mathrm{d}\theta_1
\leq |S^n|\log r_{\min}+C_3,
\end{aligned}
\end{equation}
where the second term in the last inequality is convergent for the same reason as \eqref{thm2 logv}.
Plugging \eqref{thm2 logv} and \eqref{thm2 logt} in \eqref{thm2 Q}, we obtain
\begin{equation}
\begin{aligned}
-C_1\leq C_4+f_{\max}^{-1}|S^n|\log r_{\min}-(1-c_2^2)^{-\frac{n+1}{2}}|S^n|\log\hat{u}_{\max}.
\end{aligned}
\end{equation}
Hence, there exists positive constants $c_3$ and $c_4$, such that 
\begin{align}\label{min-mix r}
r_{\min}\geq c_3 r_{\max}^{c_4},
\end{align}
where $c_4=f_{\max}(1-c_2^2)^{-\frac{n+1}{2}}$.\\
Dividing the both sides of \eqref{flow s n+1} by $r\sqrt{1-r^2}f$, we have
\begin{align}\label{thm2 r1}
\frac{\partial_t r}{r\sqrt{1-r^2}f} =-r^{n+1}\sqrt{1-r^2}\hat{u}^{-1}(1-\hat{u}^2)^{-\frac{n+1}{2}}\hat{K}+f^{-1}(\t).
\end{align}
Integrating \eqref{thm2 r1} over $\mathbb{S}^n$, we get
\begin{equation}\label{thm2 r2}
\begin{aligned}
\int_{\mathbb{S}^n}\frac{\partial_t r}{r\sqrt{1-r^2}f}\mathrm{d}\t_{\mathbb{S}^n} =&-\int_{\mathbb{S}^n}r^{n+1}\sqrt{1-r^2}\hat{u}^{-1}(1-\hat{u}^2)^{-\frac{n+1}{2}}\hat{K}\mathrm{d}\t_{\mathbb{S}^n} +\int_{\mathbb{S}^n}f^{-1}(\t)\mathrm{d}\t_{\mathbb{S}^n}\\
=& -\int_{\mathbb{S}^n}\sqrt{1-r^2}(1-\hat{u}^2)^{-\frac{n+1}{2}}\mathrm{d}\s_{\mathbb{S}^n} +\int_{\mathbb{S}^n}f^{-1}(\t)\mathrm{d}\t_{\mathbb{S}^n},
\end{aligned}
\end{equation}
where we use $r^{n+1}d\t_{\mathbb{S}^n}=\frac{\hat{u}}{\hat{K}}d\s_{\mathbb{S}^n}$ in the last equality.
Note that the left hand side of \eqref{thm2 r2} is
\begin{align}\label{thm2 r3}
\int_{\mathbb{S}^n}\frac{\partial_t r}{r\sqrt{1-r^2}f}\mathrm{d}\t_{\mathbb{S}^n}=\partial_t\displaystyle\int_{\mathbb{S}^n}\int_{c_2}^r s^{-1}(1-s^2)^{-\frac{1}{2}} f^{-1} \mathrm{d}s \mathrm{d}\theta_{\mathbb{S}^n}.
\end{align}
Besides, we have
\begin{equation}\label{2 r5}
\begin{aligned}
&\displaystyle\int_{\mathbb{S}^n}\int_{c_2}^r s^{-1}(1-s^2)^{-\frac{1}{2}} f^{-1} \mathrm{d}s \mathrm{d}\theta_{\mathbb{S}^n}\\
\leq& -f_{\max}^{-1}\displaystyle\int_{\mathbb{S}^n}\int_r^{c_2} s^{-1}\mathrm{d}s \mathrm{d}\theta_{\mathbb{S}^n}\\
=&-|S^n|f_{\max}^{-1}\log c_2+f_{\max}^{-1}\displaystyle\int_{\mathbb{S}^n}\log r \mathrm{d}\theta_{\mathbb{S}^n}\\
\leq&C_{5}+|S^n|f_{\max}^{-1}\log r_{\max} 
\end{aligned}
\end{equation}
and
\begin{equation}\label{2 r6}
\begin{aligned}
&\displaystyle\int_{\mathbb{S}^n}\int_{c_2}^r s^{-1}(1-s^2)^{-\frac{1}{2}} f^{-1} \mathrm{d}s \mathrm{d}\theta_{\mathbb{S}^n}\\
\geq &-\f 1-c_2^2\r^{-\frac{1}{2}}f_{\min}^{-1}\displaystyle\int_{\mathbb{S}^n}\int_r^{c_2} s^{-1}\mathrm{d}s \mathrm{d}\theta_{\mathbb{S}^n}\\
=&-\f 1-c_2^2\r^{-\frac{1}{2}}|S^n|f_{\min}^{-1}\log c_2+\f 1-c_2^2\r^{-\frac{1}{2}}f_{\min}^{-1}\int_{\mathbb{S}^n}\log r \mathrm{d}\theta_{\mathbb{S}^n}\\
\geq &\f 1-c_2^2\r^{-\frac{1}{2}}|S^n|f_{\min}^{-1}\log r_{\min}.
\end{aligned}
\end{equation}
Since $r(\t(v),t)=\sqrt{\hat{u}(v)^2+|\nabla \hat{u}(v)|^2}\geq\hat{u}(v,t)$ on $\mathbb{S}^n$, we have
\begin{align}\label{thm2 r4}
\int_{\mathbb{S}^n}\sqrt{1-r^2}(1-\hat{u}^2)^{-\frac{n+1}{2}}\mathrm{d}\s_{\mathbb{S}^n}\leq \int_{\mathbb{S}^n}(1-\hat{u}^2)^{-\frac{n}{2}}\mathrm{d}\s_{\mathbb{S}^n}\leq (1-\hat{u}_{\max}^2)^{-\frac{n}{2}}|S^n|.
\end{align}
When $\hat{u}_{\max}\rightarrow 0$, $(1-\hat{u}_{\max}^2)^{-\frac{n}{2}}|S^n|\rightarrow|S^n|$. By \eqref{thm2 r4} and the condition of $f$, there exists a positive constant $c_5$, such that when $u_{\max}\leq c_5$, $\int_{\mathbb{S}^n}\sqrt{1-r^2}(1-\hat{u}^2)^{-\frac{n+1}{2}}\mathrm{d}\s_{\mathbb{S}^n}\leq \frac{1}{2}(1+\int_{\mathbb{S}^n}f^{-1}(\t)\mathrm{d}\t_{\mathbb{S}^n})< \int_{\mathbb{S}^n}f^{-1}(\t)\mathrm{d}\t_{\mathbb{S}^n}$. This together with \eqref{thm2 r2} and \eqref{thm2 r3} implies
\begin{align}\label{2 r7}
\partial_t\int_{\mathbb{S}^n}\int_{c_2}^r s^{-1}(1-s^2)^{-\frac{1}{2}} f^{-1} \mathrm{d}s \mathrm{d}\theta_{\mathbb{S}^n} =-\int_{\mathbb{S}^n}\sqrt{1-r^2}(1-\hat{u}^2)^{-\frac{n+1}{2}}\mathrm{d}\s_{\mathbb{S}^n} +\int_{\mathbb{S}^n}f^{-1}(\t)\mathrm{d}\t_{\mathbb{S}^n}>0.
\end{align}
When $\hat{u}_{\max}=c_5$, inserting \eqref{min-mix r} into \eqref{2 r6}, we have 
\begin{equation}\label{2 r8}
\begin{aligned}
\displaystyle\int_{\mathbb{S}^n}\int_{c_2}^r s^{-1}(1-s^2)^{-\frac{1}{2}} f^{-1} \mathrm{d}s \mathrm{d}\theta_{\mathbb{S}^n} \geq &\f 1-c_2^2\r^{-\frac{1}{2}}|S^n|f_{\min}^{-1}\log c_3+c_4\f 1-c_2^2\r^{-\frac{1}{2}} |S^n|f_{\min}^{-1}\log r_{\max}\\
\geq& -C_6+(1-c_2^2)^{-\frac{n+2}{2}}|S^n|f_{\max}f_{\min}^{-1}\log c_5,
\end{aligned}
\end{equation}
where we use $\max\limits_{\hat{M}_t}\hat{u}=\max\limits_{\hat{M}_t}r_{\max}$. 

If the maximal radial function of the initial hypersurface satisfies $\hat{u}_{\max}(0)>c_5$, then once $\hat{u}_{\max}(t)\leq c_5$,
by \eqref{2 r5}, \eqref{2 r7} and \eqref{2 r8} we obtain
\begin{align}\label{2 r9}
\log r_{\max}\geq -C_7+(1-c_2^2)^{-\frac{n+1}{2}}f_{\max}^2f_{\min}^{-1}\log c_5,
\end{align}
which implies that there exists a positive constant $c_6=e^{-C_7}c_5^{(1-c_2^2)^{-\frac{n+1}{2}}f_{\max}^2f_{\min}^{-1}}< c_5$, such that $r_{\max}\geq c_6$, untill $\hat{u}_{\max}>c_5$ again.

If the initial hypersurface of \eqref{flow s n+1} satisfies $\hat{u}_{\max}(0)\leq c_5$, then from \eqref{2 r7}, we note that $\int_{\mathbb{S}^n}\int_a^r s^{-1}(1-s^2)^{-\frac{1}{2}} f^{-1} \mathrm{d}s \mathrm{d}\theta_{\mathbb{S}^n}$ will monotone increasing till $\hat{u}_{\max}(t)>c_5$. This together with \eqref{2 r5} implies that
\begin{align}\label{2 r10}
|S^n|f_{\max}^{-1}\log r_{\max}+C_5\geq \int_{\mathbb{S}^n}\int_{c_2}^r s^{-1}(1-s^2)^{-\frac{1}{2}} f^{-1} \mathrm{d}s\mid_{t=0}.
\end{align}
Hence there exists a positive constant $c_7$, such that $r_{\max}\geq c_7$. Clearly, we can deduce from \eqref{2 r5} that $c_7\leq\hat{u}_{\max}(0)\leq c_5$. Since $\hat{u}_{\max}=r_{\max}$ in $\hat{M}_t$, we obtain $r_{\max}(t)\geq \min\lbrace c_6, c_7\rbrace$. All those constants only depend on $n$, $f$ and the initial hypersurface $M_0$. By \eqref{min-mix r}, there exists a positive constant $c_8$, such that $r_{\min}\geq c_8$. Thus the radial function $r$ of $\hat{M}_t$ satisfying $0<c_8\leq r\leq c_2<1$. By \eqref{sinh}, we obtain the uniform bounds of $\rho$ and complete the proof. 
\end{proof}

\noindent\textbf{Proof of Theorem 1.2}  Similar to the proof of Theorem \ref{thm1}, by Lemma \ref{thm2 c0}, Lemma \ref{c1} and Lemma \ref{conv-c2}, we establish the a priori estimates and obtain the long time existence of  \eqref{flow-s} for the case $\a=n+1$. Using the same argument in Section \ref{sec:5}, we show that the flow \eqref{flow-s} converges smoothly to the unique smooth even solution of \eqref{alex2}. Hence we complete the proof of Theorem \ref{thm2}.\qed

\begin{rem}
Note that the conditions of $\a$ and $f$ in Theorem \ref{thm1} and Theorem \ref{thm2} are necessary to the convergence and asymptotic results of \eqref{flow-s}. When $\a<n+1$, or When $\a=n+1$ and $f\geq 1$ at the same time, consider a geodesic sphere with its radial function $\rho\equiv c$. Note that $\sinh\rho\rightarrow 0$ and $\cosh\rho\rightarrow 1$ as $\rho\rightarrow 0$. When $\a<n+1$,
\begin{align*}
\sinh\rho^{\a-n-1}\cosh\rho^n \rightarrow\infty.
\end{align*}
When $\a=n+1$,
\begin{align*}
\sinh\rho^{\a-n-1}\cosh\rho^n \rightarrow 1^+.
\end{align*}
Hence there exists a constant $\rho_0>0$, such that for any $\rho\in(0,\rho_0)$, $-\sinh\rho^{\a-n-1}\cosh\rho^n f+1<0$ when $\a<n+1$, or when $\a=n+1$ and $f\geq 1$ at the same time. Assume the initial hypersurface of the flow \eqref{flow-s} is a geodesic sphere with its radial function $\rho\in(0,\rho_0)$. Then combining \eqref{s1:flow-rn} ($\tilde{\eta}(t)\equiv 1$) and \eqref{c01}, we obtain
\begin{align*}
\frac{\partial}{\partial t}\rho(\theta,t)=&-\phi(\rho)^{\a}f(\theta)w\s_k +\phi(\theta,t)\\
=&\phi(\rho)\f -\phi(\rho)^{\a-n-1}\phi'(\rho)^n+1) \r\\
<& 0.
\end{align*}
So the initial geodesic sphere keeps shrinking along \eqref{flow-s}.

However, for the cases $\a> n+1$, the same initial geodesic sphere of \eqref{flow-s} will expand untill it converges smoothly to the solution of \eqref{alex2}. For the case $\a=n+1$, convergence results of the flow \eqref{flow-s} need the assumption of $f$ in addition.
\end{rem}

\section{Proof of Theorem \ref{thm3}}\label{sec:7}
In this part, we consider the flow \eqref{flow-n} for the cases $2<\a\leq n+1$. First, we parametrize $M_t$ as a graph of the radial function $\rho(\theta,t): \mathbb{S}^n\times [0,T)\to \mathbb{R}$. By \eqref{s1:flow-rn}, the scalar parabolic PDE of the radial function turns to 
\begin{equation}\label{flow-nr}
\frac{\partial}{\partial t}\rho(\theta,t)= -\phi(\rho)^{\a}f(\theta)K(x,t)w+\frac{\displaystyle\int_{\mathbb{S}^n}\frac{K}{u}\phi^{n+1}\mathrm{d}\theta_{\mathbb{S}^n}}{\displaystyle\int_{\mathbb{S}^n}\phi^{n+1-\a}f^{-1}\mathrm{d}\theta_{\mathbb{S}^n}}\phi.
\end{equation}
Throughout this section, if not specified, we will regard $C_i$ and $C_I'$ for $i\in\mathbb{N}$ as some positive constants.
Observe that along \eqref{flow-n}, the hypersurfaces have the following properties.
\begin{lem}\label{vol phi}
Denote $\Omega(\rho)=\displaystyle\int_b^{\rho(\t,t)}(\sinh s)^{n-\a}\mathrm{d}s$. Then along the flow \eqref{flow-n},
\begin{align}
\displaystyle\int_{\mathbb{S}^n}\frac{\Omega(\rho)}{f}\mathrm{d}\theta_{\mathbb{S}^n}=const,\quad\text{for } t\geq 0.
\end{align}
Here $\rho(\dot,t)$ is the radial function of $M_t$.
\end{lem}

\begin{proof}
By \eqref{flow-nr},
\begin{equation}
\begin{aligned}
\partial_t\displaystyle\int_{\mathbb{S}^n}\frac{\Omega(\rho)}{f}\mathrm{d}\theta_{\mathbb{S}^n}=&\displaystyle\int_{\mathbb{S}^n}\phi (\rho)^{n-\a}f^{-1}\partial_t\rho \mathrm{d}\theta_{\mathbb{S}^n}=0
\end{aligned}
\end{equation}
where we use $w=\frac{\phi}{u}$ in the last equality.
\end{proof}
\begin{rem}
Without loss of generality, we can choose $b=\frac{1}{2}\min\limits_{M_0}\rho$.
\end{rem}
From Lemma \ref{vol phi}, we have the following proposition.
\begin{prop}
Along \eqref{flow-nr}, the maximal radial functions $\rho_{\max}(t)$ of $M_t$ have uniform lower bounds.
\end{prop}
\begin{proof}
Observe that when $\a<n+1$, since $\cosh s>0$  and it is monotone increasing, we have
\begin{equation}\label{int phi low}
\begin{aligned}
&\phi^{n+1-\a}(\rho_{\max})|\mathbb{S}^n|-\phi^{n+1-\a}(b)|\mathbb{S}^n|\\
\geq &\int_{\mathbb{S}^n}\f\phi^{n+1-\a}(\rho)-\phi^{n+1-\a}(b)\r\mathrm{d}\theta_{\mathbb{S}^n}\\
=&(n+1-\a)\int_{\mathbb{S}^n}\int_b^{\rho}(\sinh s)^{n-\a}\cosh s \ \mathrm{d}s \mathrm{d}\theta_{\mathbb{S}^n}\\
\geq&(n+1-\a)\int_{\mathbb{S}^n}\int_b^{\rho}(\sinh s)^{n-\a}\cosh b \  \mathrm{d}s \mathrm{d}\theta_{\mathbb{S}^n}\geq c>0.
\end{aligned}
\end{equation}
Here we use Lemma \ref{vol phi} in the last inequality. The positive constant $c$ depends on the initial hypersurface, $\min f$ and $a$. Similarly,
when $\a=n+1$, there exists a positive constant $c'$, such that 
\begin{equation}\label{int phi low'}
\begin{aligned}
\log\phi(\rho_{\max})|\mathbb{S}^n|-\log\phi(b)|\mathbb{S}^n|\geq &\int_{\mathbb{S}^n}\f\log\phi(\rho)-\log\phi(b)\r\mathrm{d}\theta_{\mathbb{S}^n}\\
=&\int_{\mathbb{S}^n}\int_b^{\rho}(\sinh s)^{-1}\cosh s \ \mathrm{d}s \mathrm{d}\theta_{\mathbb{S}^n}\\
\geq&\int_{\mathbb{S}^n}\int_b^{\rho}(\sinh s)^{-1}\cosh b \ \mathrm{d}s \mathrm{d}\theta_{\mathbb{S}^n}\geq c'>0.
\end{aligned}
\end{equation}
If $\rho_{\max}\rightarrow 0$, we have $\phi(\rho_{\max})^{n+1-\a}\rightarrow 0$ and $\log\phi(\rho_{\max})\rightarrow -\infty$, which is contrary to \eqref{int phi low} and \eqref{int phi low'}. Hence, we obtain the uniform lower bound of $\phi_{\max}$ and complete the proof.

\end{proof}

Using the projection $\pi_p$ in Section \ref{sec:5} and the relation between $M_t$ and $\pi_p(M_t)$, \eqref{proj re}-\eqref{u np} and \eqref{hat k}, we obtain the scalar equation of the support function $\hat{u}$ of $\pi_p(M_t)$,
\begin{align}\label{flow np}
\partial_t \hat{u}=-r^{\a}(1-r^2)^{\frac{n+3-\a}{2}}(1-\hat{u}^2)^{-\frac{n+1}{2}}f\hat{K}+\eta(t)\hat{u}\sqrt{1-r^2},
\end{align}
where $\eta(t)=\frac{\displaystyle\int\sqrt{1-r^2}(1-\hat{u}^2)^{-\frac{n+1}{2}}\mathrm{d}\sigma_{\mathbb{S}^n}}{\displaystyle\int\frac{r^{n+1-\a}}{(1-r^2)^{\frac{n+1-\a}{2}}}f^{-1}\mathrm{d}\theta_{\mathbb{S}^n}}$.

Under the flow \eqref{flow np}, we have the following monotone function
\begin{align}
\mathcal{J}(\hat{u})=\int_{\mathbb{S}^n}\Psi(\hat{u}) \mathrm{d}\sigma_{\mathbb{S}^n}
\end{align}
where $\Psi(\hat{u})=\int_a^{\hat{u}(x,t)}\frac{1}{s}(1-s^2)^{-\frac{n+1}{2}}\mathrm{d}s$. Without loss of generality, we can choose $a=\frac{1}{2}\min\limits_{\hat{M_0}}\hat{u}$.
\begin{lem}\label{ju}
Along \eqref{flow np}, $\mathcal{J}(\hat{u})$ is non-increasing and the equality holds if and only if $\pi_p(M_t)$ satisfies the following equation 
\begin{align}\label{np cri}
r^{\a}(1-r^2)^{\frac{n+3-\a}{2}}(1-\hat{u}^2)^{-\frac{n+1}{2}}f\hat{K}=c\hat{u}\sqrt{1-r^2}
\end{align}
for some positive constant $c$.
\end{lem}

\begin{proof}
By \eqref{flow np},
\begin{equation}
\begin{aligned}
\partial_t\mathcal{J}(\hat{u})=&\int_{\mathbb{S}^n}\frac{1}{\hat{u}}(1-\hat{u}^2)^{-\frac{n+1}{2}}\partial_t u \mathrm{d}\sigma_{\mathbb{S}^n}\\
=&\frac{\f \displaystyle\int\sqrt{1-r^2}(1-\hat{u}^2)^{-\frac{n+1}{2}}\mathrm{d}\sigma_{\mathbb{S}^n}\r^2}{\displaystyle\int\frac{r^{-\a}\hat{u}}{(1-r^2)^{\frac{n+1-\a}{2}}\hat{K}}f^{-1}\mathrm{d}\sigma_{\mathbb{S}^n}}-\displaystyle\int r^{\a}(1-r^2)^{\frac{n+3-\a}{2}}(1-\hat{u}^2)^{-(n+1)}f\frac{\hat{K}}{\hat{u}}\mathrm{d}\sigma_{\mathbb{S}^n}\\
\leq& 0.
\end{aligned}
\end{equation}
Here we use H$\ddot{\text{o}}$lder inequality in the last inequality and the equality holds if and only if \eqref{np cri} holds. 

\end{proof}
Next, we will show the radial function $\rho(\cdot, t)$ is uniformly bounded along the flow \eqref{flow-n} under the following assumption.
\begin{prop}\label{thm3 c0}
Let $\rho$ be a smooth, positive, uniformly convex and origin-symmetric solution to \eqref{flow-nr} on $\mathbb{S}^n\times[0,T)$. If $2<\a\leq n+1$ and $f$ is a positive even function on $\mathbb{S}^n$, then there exists a positive constant $C$ depending on $\a$, $n$, $f$ and the initial hypersurface, such that
\begin{align}
\frac{1}{C}\leq \rho \leq C,\quad \forall t\in[0,T).
\end{align} 

\end{prop}
\begin{proof}
Firstly, we prove $\rho$ has a uniform upper bound. Motivated by the proof in Lemma \ref{C0-est}, we will show that $\frac{\displaystyle\int_{\mathbb{S}^n}{\frac{K}{u}\phi^{n+1}\mathrm{d}\theta_{\mathbb{S}^n}}}{\displaystyle\int_{\mathbb{S}^n}\phi^{n+1-\a}f^{-1}\mathrm{d}\theta_{\mathbb{S}^n}}$ is bounded from above by multiple $\phi_{\max}$.
By Lemma \ref{ju}, there exists a positive constant $C_1>0$, such that
\begin{equation}
\begin{aligned}\label{ju2}
C_1\geq &\mathcal{J}(\hat{u})\\
= &\int_{\mathbb{S}^n}\int_a^{\hat{u}}\frac{1}{s}(1-s^2)^{-\frac{n+1}{2}}\mathrm{d}s \mathrm{d}\sigma_{\mathbb{S}^n}\\
\geq &2^{\frac{n+1}{2}}\int\limits_{\mathbb{S}^n\bigcap\lbrace  x|\hat{u}>a\rbrace}\int_a^{\hat{u}}(1-s)^{-\frac{n+1}{2}}\mathrm{d}s \mathrm{d}\sigma_{\mathbb{S}^n}-(1-a^2)^{-\frac{n+1}{2}}\int\limits_{\mathbb{S}^n\bigcap\lbrace  x|\hat{u}\leq a\rbrace}\int_{\hat{u}}^a \frac{1}{s}\mathrm{d}s \mathrm{d}\sigma_{\mathbb{S}^n}\\
\geq &C_2\int\limits_{\mathbb{S}^n\bigcap\lbrace  x|\hat{u}>a\rbrace}\f(1-\hat{u})^{-\frac{n-1}{2}}-(1-a)^{\frac{n-1}{2}}\r\mathrm{d}\sigma_{\mathbb{S}^n}-C_3\int\limits_{\mathbb{S}^n\bigcap\lbrace  x|\hat{u}\leq a\rbrace}\f\log a-\log\hat{u}\r \mathrm{d}\sigma_{\mathbb{S}^n}\\
\geq& -C_4+C_2\int_{\mathbb{S}^n}(1-\hat{u})^{-\frac{n-1}{2}} \mathrm{d}\sigma_{\mathbb{S}^n}-C_2\int\limits_{\mathbb{S}^n\bigcap\lbrace  x|\hat{u}\leq a\rbrace}(1-\hat{u})^{-\frac{n-1}{2}} \mathrm{d}\sigma_{\mathbb{S}^n}+C_3\int\limits_{\mathbb{S}^n}\log\hat{u} \ \mathrm{d}\sigma_{\mathbb{S}^n}\\
\geq& -C_4+C_2\int_{\mathbb{S}^n}(1-\hat{u})^{-\frac{n-1}{2}} \mathrm{d}\sigma_{\mathbb{S}^n}+C_3\int\limits_{\mathbb{S}^n}\log\hat{u} \ \mathrm{d}\sigma_{\mathbb{S}^n}.
\end{aligned}
\end{equation}
Similar to the proof in \eqref{thm2 Q},
we have when $n=1$
\begin{align}\label{ju2'}
-C_5'\int_{\mathbb{S}^1}\log(1-\hat{u}) \mathrm{d}\sigma_{\mathbb{S}^n}+C_6'\int_{\mathbb{S}^1}\log \hat{u} \mathrm{d}\sigma_{\mathbb{S}^1}\leq C_7',
\end{align}
and when $n\geq 2$
\begin{align}\label{ju2"}
C_5\int\limits_{\mathbb{S}^n}(1-\hat{u})^{-\frac{n-1}{2}} \mathrm{d}\sigma_{\mathbb{S}^n}+C_6\int\limits_{\mathbb{S}^n}\log \hat{u} \mathrm{d}\sigma_{\mathbb{S}^n}\leq C_7.
\end{align}
Assume $\hat{u}_{\max}(t)$ is attained at $v_0\in \mathbb{S}^n$ at any fixed time $t$. Because the flow hypersurfaces are uniformly convex and origin-symmetric, we have $\hat{u}(v) \geq \hat{u}_{\max}|\metric{v}{v_0}|$. Then
\begin{align}\label{log u}
\int_{\mathbb{S}^n}\log \hat{u} \mathrm{d}\sigma_{\mathbb{S}^n}\geq \int_{\mathbb{S}^n}(\log \hat{u}_{\max}+\log |\metric{v}{v_0} |) \mathrm{d}\sigma_{\mathbb{S}^n}.
\end{align}
The second integral in \eqref{log u} is convergent, for the same reason as \eqref{thm2 logv}.
Thus by \eqref{ju2'} and \eqref{ju2"}, there exists positive constants $C_5$ and $C_5'$ such that,
when $n=1$
\begin{align}\label{ju 1'}
\int_{\mathbb{S}^1}\log(1-\hat{u}^2)\mathrm{d}\sigma_{\mathbb{S}^1}\geq\int_{\mathbb{S}^1}\log(1-\hat{u})\mathrm{d}\sigma_{\mathbb{S}^1}\geq -C_8',
\end{align} 
and when $n\geq 2$
\begin{align}\label{ju 1"}
C_8\geq\int_{\mathbb{S}^n}(1-\hat{u})^{-\frac{n-1}{2}} \mathrm{d}\sigma_{\mathbb{S}^n}\geq\int_{\mathbb{S}^n}(1-\hat{u}^2)^{-\frac{n-1}{2}} \mathrm{d}\sigma_{\mathbb{S}^n}.
\end{align} 
\eqref{int phi low} and \eqref{int phi low'} implies that there exists a constant $c_1>0$, such that $\rho_{max}\geq c_1$, i.e., by \eqref{sinh} there exists $c_2>0$, such that $r\geq c_2$. Recall the relation between $M_t$ and $\pi_p(M_t)$, \eqref{proj re}-\eqref{u np} and \eqref{hat k}. Note that $r(\t,t)=\sqrt{\hat{u}^2+|\nabla \hat{u}|^2}\geq \hat{u}(v,t)$ and $\max\limits_{\hat{M}_t}r=\max\limits_{\hat{M}_t}\hat{u}$. At a fixed time $t$, we have
\begin{equation}\label{int 3}
\begin{aligned}
\displaystyle\int_{\mathbb{S}^n}\frac{K}{u}\phi^{n+1}\mathrm{d}\theta_{\mathbb{S}^n}=&\displaystyle\int_{\mathbb{S}^n}\sqrt{1-r^2}(1-\hat{u}^2)^{-\frac{n+1}{2}}\mathrm{d}\sigma_{\mathbb{S}^n}\\
\leq &\frac{1}{\sqrt{1-\hat{u}_{\max}^2}}\int_{\mathbb{S}^n}(1-\hat{u}^2)^{-\frac{n-1}{2}}\mathrm{d}\sigma_{\mathbb{S}^n}\\
\leq &\frac{C_8}{c_2}\frac{r_{\max}}{\sqrt{1-r_{\max}^2}}=C_9\phi_{\max},
\end{aligned}
\end{equation}
where we use \eqref{ju 1"} in the second inequality and \eqref{sinh} in the last equality. Moreover, the uniform lower bound of $\displaystyle\int_{\mathbb{S}^n}\phi^{n+1-\a}f^{-1}\mathrm{d}\theta_{\mathbb{S}^n}$ is obtained  directly by \eqref{int phi low}. At the maximal point of $\rho$, inserting \eqref{int 3} into \eqref{flow-nr}, similar to the proof in Lemma \ref{C0-est}, we have
\begin{equation}
\begin{aligned}
\partial_t\rho_{\max}\leq &-\phi_{\max}^{\a-n}\phi_{\max}'^nf+\frac{\displaystyle\int_{\mathbb{S}^n}\frac{K}{u}\phi^{n+1}\mathrm{d}\theta_{\mathbb{S}^n}}{\displaystyle\int_{\mathbb{S}^n}\phi^{n+1-\a}f^{-1}\mathrm{d}\theta_{\mathbb{S}^n}}\phi_{\max}\\
\leq&\phi_{\max}^2(-\phi_{\max}^{\a-2}f+C_7).
\end{aligned}
\end{equation}
Hence when $\a>2$, $\phi=\sinh \rho$ has an uniform upper bound. Thus by \eqref{sinh} there exists a constant $c_3>0$, such that $r\leq c_3<1$.

In order to obtain the lower bound of $\rho$, we parametrize any point $\theta$ in $\mathbb{S}^n$ as in \eqref{theta para}. Assume $r_{\min}$ is attained at $\theta_0=(1,\vec{0})$. We have $\metric{\theta}{\theta_0}=\cos \theta_1$.
Since $r\leq c_3<1$, there exists a positive constant $c_4$, such that $ \phi=\frac{r}{\sqrt{1-r^2}}\leq c_4 r$. Denote by $\delta=\sqrt{r_{\min}}$.
For $2<\a<n+1$, we have
\begin{equation}\label{c0 thm3}
\begin{aligned}
\int_{\mathbb{S}^n}\phi^{n+1-\a}(\rho)\mathrm{d}\theta_{\mathbb{S}^n}\leq& c_4^{n+1-\a}\int\limits_{S_1=\mathbb{S}^n\bigcap\lbrace r\leq \delta\rbrace}r^{n+1-\a}\mathrm{d}\theta_{\mathbb{S}^n}+c_4^{n+1-\a}\int\limits_{S_2=\mathbb{S}^n\bigcap\lbrace r>\delta\rbrace}r^{n+1-\a}\mathrm{d}\theta_{\mathbb{S}^n}\\
\leq&C_8\delta^{n+1-\a}|S^n|+C_8c_3^{n+1-\a}|S_2|.
\end{aligned}
\end{equation}
 We also have $r(\theta)\leq \frac{r_{\min}(\theta_0)}{|\cos \theta_1|}$ because the flow hypersurface $M_t$ are strictly convex and origin symmetric. Then $S_2=\lbrace \theta|\cos\theta_1<\sqrt{r_{\min}}\rbrace$. Suppose $r_{\min}\rightarrow 0$. Clearly, $\delta=\sqrt{r_{\min}}\rightarrow 0$ and $|S_2|\rightarrow 0$ at the same time. By \eqref{c0 thm3}, we have $\int_{\mathbb{S}^n}\phi^{n+1-\a}(\rho)\mathrm{d}\theta_{\mathbb{S}^n}\rightarrow 0$, which is contrary to \eqref{int phi low}.
For $\a=n+1$, we have
\begin{equation}\label{thm3 f}
\begin{aligned}
\int_{\mathbb{S}^n}\log\phi(\rho)\mathrm{d}\theta_{\mathbb{S}^n}
\leq &\int_{\mathbb{S}^n}\log r(\theta)\mathrm{d}\theta_{\mathbb{S}^n}+|S^n|\log c_4\\
\leq &|S^n|\log r_{\min}(\theta_0)-\int_{\mathbb{S}^n}\log|\cos \theta_1| \mathrm{d}\theta_{\mathbb{S}^n}+|S^n|\log c_4.
\end{aligned}
\end{equation}
For the same reason as \eqref{thm2 logt}, the second term in \eqref{thm3 f} is convergent. So $\int_{\mathbb{S}^n}\log\phi(\rho)\mathrm{d}\theta_{\mathbb{S}^n}\rightarrow -\infty$ as $r_{\min}\rightarrow 0$, which is contrary to \eqref{int phi low'}. Hence we obtain the uniform lower bounds of $r$ as well as $\rho$ by \eqref{sinh} and complete the proof.
\end{proof}
\begin{rem}
Similar to the Euclidean space, we hope \eqref{ju2} implies the uniform upper bound of the support function $\hat{u}$ away from $1$. Unfortunately, we find an example here with unbounded $u$, i.e, $\hat{u}$ is very close to $1$ satisfying \eqref{ju2}. Suppose $\hat{E}(e_1,e_2,\cdots,e_2)$ is a rotation symmetric ellipsoid in $\mathbb{R}^{n+1}$ with its centroid in the origin represented by $\frac{x_1^2}{e_1^2}+\frac{x_2^2}{e_2^2}+\cdots+\frac{x_{n+1}^2}{e_2^2}=1$. Here we assume $0<e_2<e_1<1$. Then $\hat{E}$ can be parametrized by $x=(e_1\cos \theta,e_2\sin\theta\vec{x})$ where $\vec{x}\in \mathbb{S}^{n-1}$ is an $n$-dimensional unit vector. Correspondingly, the unit outward norm vector is 
\begin{align}\label{E u}
\frac{(e_2\cos \theta,e_1\sin\theta\vec{x})}{\sqrt{e_2^2\cos^2\theta+e_1^2\sin^2\theta}}.    
\end{align}
The support function becomes
\begin{align}\label{utheta}
\hat{u}(\theta)=\frac{e_1e_2}{\sqrt{e_2^2\cos^2\theta+e_1^2\sin^2\theta}}.
\end{align}
If we parametrize $\mathbb{S}^n$ as in \eqref{nu para}, by comparing with \eqref{E u} we have 
\begin{align}\label{nt}
\tan v_1=\frac{e_1}{e_2}\tan\theta.
\end{align}  
By \eqref{utheta} and \eqref{nt}, the support function can be parametrized by $v_1$ as 
\begin{equation}\label{unu}
\begin{aligned}
\hat{u}(v_1)=&\frac{e_1e_2\sqrt{\tan \theta^2+1}}{\sqrt{e_2^2+e_1^2\tan\theta^2}}=\frac{e_1\sqrt{\tan v_1^2\frac{e_2^2}{e_1^2}+1}}{\sqrt{1+\tan v_1^2}}\\
=&\sqrt{\tan v_1^2e_2^2+e_1^2}\cos v_1=\sqrt{e_1^2-(e_1^2-e_2^2)\sin v_1^2}.
\end{aligned}
\end{equation}
Substitute \eqref{nu para} into \eqref{ju 1'} and \eqref{ju 1"} respectively. When $n=1$, we obtain
\begin{equation}\label{umax1 1}
\begin{aligned}
\int_{\mathbb{S}^1}\log(1-\hat{u}^2)\mathrm{d}\sigma_{\mathbb{S}^1}=&\int_0^{2\pi}\log(1-e_1^2+(e_1^2-e_2^2)(\sin v_1)^2)dv_1\\
>&2\int_0^{\pi}\f\log(e_1^2-e_2^2)+2\log\sin v_1\r \mathrm{d}v_1\geq -C(e_1,e_2).
\end{aligned}
\end{equation}
When $n\geq 2$
\begin{equation}\label{umax1}
\begin{aligned}
&\int_{\mathbb{S}^n}(1-\hat{u}^2)^{-\frac{n-1}{2}}\mathrm{d}\sigma_{\mathbb{S}^n}\\
=&\int_0^{\pi}\int_{\mathbb{S}^{n-1}}(1-e_1^2+(e_1^2-e_2^2)\sin v_1^2)^{-\frac{n-1}{2}}(\sin v_1
)^{n-1}\mathrm{d}v_1\mathrm{d}\sigma_{\mathbb{S}^{n-1}}\\
<&\int_0^{\pi}\int_{\mathbb{S}^{n-1}}\f(e_1^2-e_2^2)\sin v_1^2\r^{-\frac{n-1}{2}}(\sin v_1
)^{n-1}\mathrm{d}v_1\mathrm{d}\sigma_{\mathbb{S}^{n-1}}\\
\leq& C(e_1,e_2,n).
\end{aligned}
\end{equation}
Now we construct a sequence of $\hat{E}(e_1^{(n)},e_2^{(n)})$ satisfying $e_1^{(n)}\rightarrow 1$ and $e_2^{(n)}\leq c<1$. Note that by \eqref{u np}, $\max u_n\rightarrow\infty$ as $\max \hat{u}_n=e_1\rightarrow 1$, which implies that the preimages of $\hat{E}(e_1^{(n)},e_2^{(n)})$ in Hyperbolic space tend to $\infty$. However, by inserting \eqref{umax1 1}, \eqref{umax1} into \eqref{ju2'}and \eqref{ju2"}, we show that $\mathcal{J}(\hat{u})_{\hat{E}(e_1^{(n)},e_2^{(n)})}$ still can be bounded from above by some positive constant $C$ independent of how far $e_1$ is away from $1$.
\end{rem}

\noindent\textbf{Proof of Theorem \ref{thm3}} Combining Lemma \ref{c1}, we established the $C^0$, $C^1$ estimates of the flow \eqref{flow-n} under the evenness assumption. Employing \eqref{sinh} and \eqref{u np}, we obtain the uniform upper and lower bounds of $\eta(t)$ in \eqref{flow np}. Then the $C^2$ estimate follows from Lemma \ref{conv-c2}. Thus we obtain the long time existence and regularity for the flow \eqref{flow-n}. Furthermore, Lemma \ref{ju} and the similar argument in Section \ref{sec:5} imply the asymptotic behaviour of the flow \eqref{flow-n} and we complete the proof of the Theorem \ref{thm3}. \qed

\end{document}